\documentclass{article}
\usepackage{amsfonts}
\usepackage{amsmath}
\usepackage{sectsty}
\usepackage{latexsym}
\usepackage{amsbsy}
\usepackage{amssymb}
\usepackage{verbatim}
\usepackage{amsthm}
\usepackage{amsfonts}
\usepackage{amsmath}
\usepackage{makeidx}
\usepackage{graphicx}
\usepackage{layout}
\usepackage{pb-diagram}
\usepackage{amscd}
\usepackage{anysize}
\usepackage{tikz}
\usepackage{graphicx}
\usepackage{upgreek}
\usepackage[all,cmtip]{xy}
\usepackage{youngtab}
\usepackage{hyperref}
\usepackage{color}
\usepackage{mathrsfs}

\newcommand{\beq}{\begin{equation*}}
\newcommand{\eeq}{\end{equation*}}

\newcommand{\calP}{\mathcal{P}}

\newcommand{\calU}{\mathcal{U}}
\newcommand{\Z}{\mathbb{Z}}
\newcommand{\Q}{\mathbb{Q}}

\newcommand{\K}{\mathbb{K}}

\newcommand{\D}{\mathbb{D}}

\DeclareMathOperator{\Th}{Th}
\DeclareMathOperator{\Diff}{Diff}
\DeclareMathOperator{\Map}{Map}
\DeclareMathOperator{\Hom}{Hom}
\DeclareMathOperator{\Inv}{Inv}
\newtheorem{theorem}{Theorem}[section]

\newtheorem{definition}[theorem]{Definition}
\newtheorem{example}[theorem]{Example}

\newtheorem{lemma}[theorem]{Lemma}

\newtheorem{proposition}[theorem]{Proposition}
\newtheorem{remark}[theorem]{Remark}

\title{\Large\bf Involution on pseudoisotopy spaces and the space of nonnegatively curved metrics}
\author{Mauricio Bustamante\and Francis Thomas Farrell\and Yi Jiang \thanks{The third author's research is partially supported by NSFC 11571343 and NSFC 11801298.}}
\newcommand{\Addresses}{{% additional braces for segregating \footnotesize
  \bigskip
  \footnotesize
  \textsc{Mauricio Bustamante}, \textsc{Department of Pure Mathematics and Mathematical Sciences, University of Cambridge, UK}\par\nopagebreak
 \texttt{bustamante.math@gmail.com}
  \medskip

  \textsc{Francis Thomas Farrell}, \textsc{Yau Mathematical Sciences Center,
  Tsinghua University, Beijing, China}\par\nopagebreak
  \texttt{farrell@math.binghamton.edu}
  \medskip

  \textsc{Yi Jiang}, \textsc{Yau Mathematical Sciences Center, Tsinghua
  University, Beijing, China}\par\nopagebreak
\texttt{yjiang117@mail.tsinghua.edu.cn}
}}
\date{}

\def\D{\underline{\calP}%
_{k}^{l,m}(M\times J^{l+m-d})}

\def\Dd{\underline{\calP}%
_{k}^{l,m}(M)}

\def\bd{\partial}

\begin{document}
\maketitle

\begin{abstract}
We prove  that certain involutions defined by Vogell and Burghelea-Fiedorowicz on the rational algebraic $K$-theory of spaces coincide. This gives a way to compute the positive and negative eigenspaces of the involution on rational homotopy groups of pseudoisotopy spaces from the involution on rational $S^{1}$-equivariant homology group of the free loop space of a simply-connected manifold. As an application, we give explicit dimensions of the open manifolds $V$ that appear in Belegradek-Farrell-Kapovitch's work for which the spaces of complete nonnegatively curved metrics on $V$ have nontrivial rational homotopy groups.

\begin{description}
\item[2010 Mathematics subject classification:] 19D10 (55N91)
\end{description}
\end{abstract}

\section{Introduction}
In this paper, all manifolds are smooth i.e. $C^{\infty}$, and any set of diffeomorphisms or Riemannian metrics is equipped with the smooth compact-open topology.

Given a compact smooth manifold $M$ possibly with boundary $\bd M$, a pseudoisotopy of $M$ is defined as a diffeomorphism $M\times [0,1]\to M\times [0,1]$ which fixes the subspace $M\times\{0\}\cup\partial M\times [0,1]$ pointwise. The space consisting of all pseudoisotopies of $M$ is denoted by $P(M)$. A feature of $P(M)$ is that it has an involution
$\iota:P(M)\rightarrow P(M)$ given essentially by reflection in the second coordinate (see Section \ref{InvPseudo} for definitions).
This involution plays an important role in understanding the homotopy type of the group $\Diff(M,\partial M)$ of diffeomorphisms of $M$ that fix $\bd M$ pointwise \cite[Proposition 2.1]{Hatcher78}.

There is also the \textit{stable pseudoisotopy space} $\mathscr{P}(M)$ which is defined as the colimit of the maps $P(M\times [0,1]^k)\to P(M\times [0,1]^{k+1})$ given by crossing a pseudoisotopy with the identity on $[0,1]$ (and smoothing corners).
A celebrated theorem of Igusa states that there is an isomorphism
\beq
\pi_{k}\mathscr{P}(M)\cong\pi_{k}P(M)
\eeq
when $\dim M>>k$. The involution on $P(M)$ can then be used to produce an involution
\beq
\iota^S:\pi_{*}\mathscr{P}(M)\to\pi_{*}\mathscr{P}(M)
\eeq
on $\pi_{*}\mathscr{P}(M)$ (see Section \ref{InvPseudo} for details).

The advantage of passing to the stable range, is that $\mathscr{P}(M)$ is related to Waldhausen's $K$-theory $A(M)$
of $M$ \cite{WaldhausenI1978}.  Moreover, the functor $A(M)$ can also be endowed with involutions which potentially serve to understand $\iota^S$. For example, given a $d-$spherical fibration $\eta$ over $M$(with a section), Vogell \cite{Vogelllong} defines an involution
\beq
\tau_{\eta}: A(M)\to A(M).
\eeq
It turns out that the involution corresponding to the sphere bundle associated to the stable tangent bundle of $M$ is compatible with $\iota^S$. The involution $\tau_{\varepsilon}$ corresponding to the trivial fibration $\varepsilon=M\times S^0$  will also play an important role in this work.

If one is willing to restrict only to $1$-connected manifolds $M$ and pass to rational homotopy groups, then $A(M)$ can further be studied via the rational $S^{1}-$equivariant homology $H_{*}^{S^{1}}(LM;\mathbb{Q})$ of the free loop space $LM$ of $M$ \cite{Waldhausenmfd, WaldhausenI1978, B, BF, Goodwillie85, Goodwillie86, Vogelllong, BF2, Lodder90, Lodder96, KS}. Note that these homology groups come equipped with a geometric involution obtained by ``reversing loops'' (see Section \ref{Section:A} for the definition.)

In this paper we establish formulas for the dimensions of the positive and negative eigenspaces of the involution $\iota^S$ on $\pi_{\ast}\mathscr{P}(M)\otimes\mathbb{Q}$ in terms of the involution $\tau_{\varepsilon}$ on the Waldhausen's $K$-theory $A(M)$ and also in terms of the involution on rational $S^{1}$-equivariant homology $H_{\ast}^{S^{1}}(LM,\ast;\mathbb{Q})$, when the smooth manifold $M$ is simply connected. To be more explicit, for a $\mathbb{Q}-$vector space $V$ with an involution $T$, let
\begin{equation*}
\Inv_{T}^{\sigma }V=\left\{ v\in V|T(v)=\sigma v\right\}
\end{equation*}%
where $\sigma =\pm $. Sometimes we omit $T$ from the notation if there is no risk of confusion.  Denote $\pi _{i}(-)\otimes\mathbb{Q}$ by $\pi _{i}^{\mathbb{Q}}(-)$. We now state our main result.

\begin{theorem}\label{formula}
Let $M$ be a simply-connected compact smooth manifold. Then for $i\geq 0$
\begin{eqnarray*}
\dim \Inv^{+}\pi _{i}^{\mathbb{Q}}\mathscr{P}(M) &=& \dim \Inv_{\tau_{\varepsilon\ast}}^{-}\pi_{i+2}^{\mathbb{Q}}A(M) \\
&=& \delta_{i}+\dim
\Inv^{+}H_{i+1}^{S^{1}}(LM,\ast ;\mathbb{Q}) \\
\dim \Inv^{-}\pi _{i}^{\mathbb{Q}}\mathscr{P}(M) &=&\dim \Inv_{\tau_{\varepsilon\ast}}^{+}\pi^{\mathbb{Q}}_{i+2}A(M)-\dim H_{i+2}(M;\mathbb{Q}) \\
&=&\dim
\Inv^{-}H_{i+1}^{S^{1}}(LM,\ast ;\mathbb{Q})-\dim H_{i+2}(M;\mathbb{Q})
\end{eqnarray*}
where $\delta_{i}=\left\{
      \begin{array}{ll}
        1, & \mbox{if }i\equiv3\mod4\mbox{;} \\
        0, & \mbox{otherwise.}
      \end{array}
    \right.$
\end{theorem}

\begin{remark}
By \cite{Dwyer}, $\dim\pi _{i}^{\mathbb{Q}}A(M)$ is finite for all $i$ if $M$ is simply-connected and $\pi_{i}(M)$ is finitely generated for each $i$.
\end{remark}

\begin{remark}
  It had already been suggested by Burghelea in his survey paper \cite[Theorem 3.5]{B89} that the involution on $\pi_{\ast}^{\mathbb{Q}}A(M)$ could be computed from the involution on $H^{\ast}_{S^{1}}(LM;\mathbb{Q})$.
\end{remark}

\begin{remark}
The comparison between the various involutions considered in this paper happen only at the level of homotopy groups. To our knowledge, a ``space'' version of the results of Vogell, namely that natural involutions on $A$-theory and pseudoisotopy theory are compatible, has not been carried out yet in the literature.
\end{remark}

Theorem \ref{formula} can be applied to the study the topology of spaces of Riemannian metrics with curvature bounds. In our case, we can draw conclusions about the rational homotopy groups of the space $\mathcal{R}_{K \geq 0}(V)$ of complete Riemannian metrics of nonnegative sectional curvature on an open manifold $V$.
A manifold $V$ supporting such metric contains a ``soul'' $S$, i.e. a totally geodesic closed submanifold of $V$ whose normal bundle is diffeomorphic to $V$. For a tubular neighborhood $N$ of a soul one can consider the map $i_N:P(\partial N)\to\Diff(N)$ given by extending a pseudoisotopy from a fixed collar neighborhood to a diffeomorphism of $N$. It turns out that non-trivial elements in the kernel of the map induced by $i_N$ on homotopy groups can be used to obtain non-trivial classes in $\pi_*\mathcal{R}_{K \geq 0}({\rm{int}}N)$ \cite[Theorem 1.1.]{BFK}. Understanding this kernel in the ``Igusa stable range'' involves surgery theory and algebraic $K$-theory of spaces. Indeed, Belegradek, Farrell and Kapovitch show that there exist manifolds $N$ as above of some dimension $m$ for which the kernel of $i_N$ contains elements of infinite order. However, their method does \textit{not} give an explicit $m$ for which this happens \cite[Remark 9.6]{BFK}. The motivation of this paper comes from trying to determine those dimensions. As will be shown in the last section, the problem reduces to being able to determine the dimensions of the positive and negative eigenspaces of the involution $\iota^S$ on the rational homotopy groups of the stable pseudoisotopy spaces of the manifolds in question, which can be done using our Theorem \ref{formula}.

The remaining parts of the paper are organized as follows. Sections 2 and 3 aim at proving Theorem \ref{formula}. In Section 2, we present the relation between the involution $\tau_{\varepsilon}$ on $A(M)$ and the geometric involution on $H_{\ast}^{S^{1}}(LM;\mathbb{Q})$. In Section 3, after reviewing Waldhausen and Vogell's work on the $K$-theory of spaces, we show the relation between the involution $\tau_{\varepsilon}$ on $A(M)$ and the involution on $\pi_{\ast}\mathscr{P}(M)$ , and then we prove Theorem \ref{formula}. In Section 4, we first calculate the involution on the rational $S^{1}$--equivariant homology groups for $LM$ when $M$ is the unit tangent bundle of an even dimensional sphere, and then apply this computation to the framework of \cite{BFK}.
\subsection*{Acknowledgements}
We thank Kristian Moi for many helpful comments and discussions. In particular, for spotting a mistake (and providing us with way to fix it) that appeared in Section 3 of  an earlier version of this paper. M.B. is grateful to Manuel Krannich for useful conversions about Section \ref{InvPseudo} of this article.
\section{Involutions on $\pi_*A(M)$ and $H_{\ast}^{S^{1}}(LM;\mathbb{Q})$}\label{Section:A}
Waldhausen \cite{Waldhausen100} defines the $K$-theory $A(M)$ of a space $M$ as the algebraic $K$-theory of the  category
$\mathcal{R}_{hf}(M)$ of retractive spaces over $M$, which is a category with cofibrations and weak equivalences. An object in $\mathcal{R}_{hf}(M)$ is a triple $(Y,r,s)$ where $r:Y\to M$ is a retraction, $s:M\to Y$ is a section to $r$ and $Y$ is a space with the homotopy type of a finite CW--complex relative to the subspace $s(M)$. Morphisms in $\mathcal{R}_{hf}(M)$ are continuous maps compatible with retractions and sections. Given a $d-$spherical fibration $\eta$ over $M$ (with a section), Vogell \cite{Vogelllong} considers Spanier-Whitehead duality in the categories of retractive spaces to define an involution $\tau_{\eta}:A(M)\to A(M)$.

In this section we establish Theorem \ref{AthS1}, which relates the involution $\tau_{\varepsilon}$ on $A(M)$ associated to the trivial fibration $M\times S^0\to M$,  with a geometric involution on $H_{\ast}^{S^{1}}(LM;\mathbb{Q})=H_{\ast}(ES^{1}\times_{S^{1}} LM;\mathbb{Q})$, where $ES^{1}\times_{S^{1}} LM$ is the Borel construction and $LM=\Map(S^1,M)$ is the free loop space  of $M$, with
$S^1$-action given by
$(z\cdot f)(w) = f(z\cdot w)$
for $z,w\in S^{1}\subset \mathbb{C}$ and $f:S^{1}\rightarrow M$. The geometric involution on $H_{i}^{S^{1}}(LM,\ast ;\mathbb{Q})$ is induced by the involution
\begin{equation*}
T: ES^{1}\times _{S^{1}}LM\rightarrow ES^{1}\times _{S^{1}}LM:[e,f]\longmapsto
\lbrack \overline{e},\overline{f}]
\end{equation*}
where $ES^{1}$ is modeled here by the infinite-dimensional sphere $S^{\infty}=\cup_{n}S^{2n-1}\subset\mathbb{C}^{\infty}=\cup_{n}\mathbb{C}^{n}$. The free action of $S^1$ on $ES^1$ is given by complex multiplication. Also $\overline{e}$ is the complex conjugate of $e\in ES^{1}\subset\mathbb{C}^{\infty }$ and $\overline{f}(x):=f(\overline{x})$ for $x\in S^{1}$.

Let $\ast\in M$ and let $H_{i}^{S^{1}}(LM,\ast;\mathbb{Q})=H_{i}(ES^{1}\times_{S^{1}}LM,ES^{1}\times_{S^{1}}\ast;\mathbb{Q})$, then the involution $T$ on the pair of spaces $(ES^{1}\times_{S^{1}}LM,ES^{1}\times_{S^{1}}\ast)$ induces an involution $T_{\ast}$ on $H_{i}^{S^{1}}(LM,\ast;\mathbb{Q})$. Recall that $A(-)$ is a functor from the category of continuous maps of topological spaces to itself \cite{WaldhausenI1978}. Since the constant map $M\rightarrow \ast$ induces a retraction $A(M)\rightarrow A(\ast)$, then the inclusion $\ast\rightarrow M$ induces a monomorphism $A(\ast)\rightarrow A(M)$. Furthermore, since the involution $\tau_{\varepsilon}$ is a natural transformation, it restricts to the involution on $A(\ast)$ and hence we have an involution $\tau_{\varepsilon\ast}$ on $\pi_{i}(A(M),A(\ast))$. Burghelea \cite{B}, Burghelea--Fiedorowicz \cite{BF}, Goodwillie \cite{Goodwillie85,Goodwillie86} and Waldhausen \cite{WaldhausenI1978} prove that $\pi_{i+1}(A(M),A(\ast ))\otimes\mathbb{Q}\cong H_{i}^{S^{1}}(LM,\ast ;\mathbb{Q})$ for all $i$. We can further obtain the following theorem.

\begin{theorem}\label{AthS1}
\label{InvonA}For a simply-connected compact manifold $M$, the isomorphism
\begin{equation*}
\pi _{i+1}(A(M),A(\ast ))\otimes\mathbb{Q}\cong H_{i}^{S^{1}}(LM,\ast ;\mathbb{Q})
\end{equation*}
can be chosen to be anti-equivariant with respect to the involutions $\tau_{\varepsilon\ast}$ and $T_{\ast}$. That is, there is an isomorphism $\pi _{i+1}(A(M),A(\ast ))\otimes\mathbb{Q}\rightarrow H_{i}^{S^{1}}(LM,\ast ;\mathbb{Q})$ such that the following diagram commutes
$$\xymatrix{\pi _{i+1}(A(M),A(\ast ))\otimes\mathbb{Q}\ar[r]\ar[d]_{\tau_{\varepsilon\ast}}&H_{i}^{S^{1}}(LM,\ast ;\mathbb{Q})\ar[d]_{-T_{\ast}}\\
\pi _{i+1}(A(M),A(\ast ))\otimes\mathbb{Q}\ar[r]&H_{i}^{S^{1}}(LM,\ast ;\mathbb{Q})}$$
\
\end{theorem}

The proof of Theorem \ref{AthS1} is given at the end of the section. The idea of the proof is the following. Let $X$ be a simply-connected simplicial set whose geometric realization $\left\vert
X\right\vert $ is homotopy equivalent to $M$. Let $K_{i}(\mathbb{Z}(G(X)))$ be the $i-$th $K$-theory group of the simplicial group ring $\mathbb{Z}[G(X)]$ where $G(X)$ is the Kan loop group of $X$ (see Section \ref{BFInvVInv} and \cite[p.276]{GJ}). Denote by $\widetilde{K}_{\ast}(\mathbb{Z}[G(X)])$ the cokernel of the natural map $K_{\ast }(\mathbb{Z})\rightarrow K_{\ast }(\mathbb{Z}[G(X)])$. Waldhausen \cite{WaldhausenI1978} has proved that there is an isomorphism
\begin{equation}\label{isomorphism:AK}
\pi _{i}(A(\left\vert X\right\vert ),A(\ast ))\otimes\mathbb{Q}\cong \widetilde{K}_{i}(\mathbb{Z}[G(X)])\otimes\mathbb{Q}
\end{equation}
and Burghelea\cite{B} and Goodwillie\cite{Goodwillie86} proved that
\begin{equation}\label{isomorphism:KS1}
\widetilde{K}_{i+1}(\mathbb{Z}[G(X)])\otimes\mathbb{Q}\cong H_{i}^{S^{1}}(L\left\vert X\right\vert ,\ast ;\mathbb{Q}).
\end{equation}
In \cite{BF2}, Burghelea and Fiedorowicz defined an involution $\tau_{BF\ast}$ on $\widetilde{K}_{\ast }(\mathbb{Z}[G(X)])$ (see Section \ref{Section:BFInv} for details). Our
strategy to prove Theorem \ref{InvonA} is to show that the isomorphisms (\ref{isomorphism:AK}) and (\ref{isomorphism:KS1}) can be chosen to be
equivariant and anti-equivariant, respectively, with respect to the involutions $\tau_{\varepsilon\ast}$, $\tau_{BF\ast}$ and $T_{\ast}$.

\subsection{The relation between the involutions on $A(\left\vert X \right\vert)$
and $K_{\ast }(\mathbb{Z}[G(X)])$}\label{BFInvVInv}
Throughout this section, let $X$ be a connected simplicial set (not necessarily simply connected). There is a linearization map $A(\left\vert X\right\vert )\rightarrow K(\mathbb{Z}[G(X)])$ which is known to be a rational homotopy equivalence \cite[Proposition 2.2]{WaldhausenI1978}. The content of this section is to compare the involutions $\tau_{\varepsilon}$ on $A(\left\vert X \right\vert)$ and $\tau_{BF}$ on $K(\Z[GX])$ under the linearization map. Note though, that the involutions $\tau_{BF}$ and $\tau_{\varepsilon}$ are of different nature: $\tau_{BF}$ is defined in terms of matrices and $\tau_{\varepsilon}$ is defined in the spirit of Spanier-Whitehead duality. Thus we should carry out this comparison in a rather indirect fashion.
The idea is to work with another model for the $K$-theory of the simplicial ring $R_{\bullet}=\Z[G(X)]$, equip it with an involution $\tau_V$ and then show on one hand that it is compatible with $\tau_{\varepsilon}$ under the linearization map, and on the other hand that it agrees with $\tau_{BF}$. We will recall the involutions $\tau_{BF}$ and $\tau_V$ on $K(\mathbb{Z}[G(X)])$ defined by Burghelea-Fiedorowicz and Vogell, respectively, in Sections \ref{Section:BFInv} and \ref{Section:VInv}. In Section \ref{Section:BFInvVInv}, we prove that the two involutions coincide and  deduce the following theorem.
\begin{theorem}\label{Theorem:AinvKinv}
Let $X$ be a connected simplicial set. For each $i\geq 0$, there is an isomorphism $\pi
_{i}(A(\left\vert X \right\vert),A(\ast))\otimes\mathbb{Q}\rightarrow\widetilde{K}_{i}(\mathbb{Z}[G(X)])\otimes\mathbb{Q}$ such that the following diagram commutes:
$$\xymatrix{\pi _{i}(A(\left\vert X \right\vert),A(\ast ))\otimes\mathbb{Q}\ar[r]\ar[d]_{\tau_{\varepsilon\ast}}&\widetilde{K}_{i}(\mathbb{Z}[G(X)])\otimes\mathbb{Q}\ar[d]_{\tau_{BF\ast}}\\
\pi _{i}(A(\left\vert X \right\vert),A(\ast ))\otimes\mathbb{Q}\ar[r]&\widetilde{K}_{i}(\mathbb{Z}[G(X)])\otimes\mathbb{Q}}$$
\end{theorem}

\subsubsection{Geometric realization of simplicial functors}

In order to introduce the involutions defined by Vogell and Burghelea-Fiedorowicz, we
recall from \cite[Section 5.2]{HKVWW} a general method to induce a map between topological spaces out of
a (contravariant) functor of simplicial categories.
Given a contravariant functor $F:C\rightarrow D$ between small categories, it induces a map $N_{\bullet}F$
from the nerve $N_{\bullet}C$ to $N_{\bullet}D$ where $N_{n}F:N_{n}C%
\rightarrow N_{n}D$ is given by $\left( g_{1},\cdots ,g_{n}\right) \mapsto
\left( F(g_{n}),\cdots ,F(g_{1})\right) $ for each $n\geq 1$ and composable
morphisms $g_{i}$ in the category $C$. Let $|C|$ denote the classifying
space of the category $C$, i.e., the geometric realization of the nerve $%
N_{\bullet}C$. Then the map $N_{\bullet}F$ is anti-simplicial (i.e. $s_{j}\circ N_{n}F=N_{n+1}F\circ s_{n-j}$ and $d_{j}\circ N_{n}F=N_{n-1}F\circ
d_{n-j}$ for $n\geq 1$ and $j\leq n$) and hence induces a map $|C|\rightarrow |D|$ via
\begin{equation*}
\begin{array}{ccc}
N_{n}C\times |\Delta ^{n}| & \longrightarrow & N_{n}D\times |\Delta ^{n}| \\
\left( x,s\right) & \mapsto & \left( N_{n}F(x),\Phi _{n}(s)\right)%
\end{array}%
\end{equation*}%
where $\Phi _{n}:|\Delta ^{n}|\longrightarrow |\Delta ^{n}|$ is the
simplicial homeomorphism of the geometric realization $|\Delta ^{n}|$
of the standard simplex $\Delta ^{n}$ which reverses the order of the
vertices. More generally, let $F_{\bullet}:C_{\bullet}\rightarrow D_{\bullet}$
be a simplicial contravariant functor of simplicial categories. As $%
F_{n}:C_{n}\rightarrow D_{n}$ is a contravariant functor of categories for
each dimension $n$, repeating the previous construction in each dimension $n$ gives rise to a map $N_{\bullet
}F_{\bullet}$ from the bisimplicial set $N_{\bullet}C_{\bullet}$ to $N_{\bullet
}D_{\bullet}$, which is antisimplicial in the first index and simplicial in
the second index. Let $|C_{\bullet}|$ denote the classifying space of the
simplicial category $C_{\bullet}$, namely, the double geometric realization
of the nerve $N_{\bullet}C_{\bullet}$ (The double geometric realization is
homeomorphic to the geometric realization of the diagonal of the
bisimplicial set, see \cite[p.94, Lemma]{Quillen}). Then $N_{\bullet}F_{\bullet}$
induces a map $|F_{\bullet}|:|C_{\bullet}|\rightarrow |D_{\bullet}|$ via the map%
\begin{equation*}
\begin{array}{ccc}
N_{n}C_{k}\times |\Delta ^{n}|\times |\Delta ^{k}| & \longrightarrow &
N_{n}D_{k}\times |\Delta ^{n}|\times |\Delta ^{k}| \\
\left( x,s,t\right) & \mapsto & \left( N_{n}F_{k}(x),\Phi _{n}(s),t\right)%
\end{array}%
\end{equation*}
In summary, we have the following lemma.

\begin{lemma}
\label{simplicialfunctor}Every simplicial contravariant functor $F_{\bullet}:C_{\bullet}\rightarrow D_{\bullet}$ of simplicial small categories induces a map $|F_{\bullet}|:|C_{\bullet}|\rightarrow |D_{\bullet}|$ between
the classifying spaces of the simplicial categories in a natural way. In particular, every anti-involution  $C_{\bullet}\rightarrow C_{\bullet}$ (i.e. a simplicial contravariant functor whose square is the identity functor) induces an involution $\vert C_{\bullet}\vert\rightarrow \vert C_{\bullet}\vert$.
\end{lemma}

\subsubsection{Burghelea-Fiedorowicz's involution}\label{Section:BFInv}
Let $R_{\bullet}$ denote the simplicial group ring $\Z[G(X)]$  generated by the Kan loop group $G(X)$ of $X$. Waldhausen \cite{WaldhausenI1978} defines the K-theory $K(R_{\bullet})$ of $R_{\bullet}$ and Burghelea and Fiedorowicz  \cite{BF2} define an involution $\tau_{BF}$ on $K(R_{\bullet})$ as follows.  Consider $\pi
_{0}:R_{\bullet}\rightarrow \pi _{0}(R_{\bullet})$ as a map of simplicial rings and define the
simplicial monoid $\widehat{GL}_{n}(R_{\bullet})$ by the pull back diagram
$$\xymatrix{\widehat{GL}_{n}(R_{\bullet})\ar[r]\ar[d] & M_{n}(R_{\bullet})\ar[d]\\
GL_{n}(\pi _{0}R_{\bullet})\ar[r] & M_{n}(\pi _{0}R_{\bullet})}$$
where $M_{n}(R_{\bullet})$ is the simplicial ring of the $n\times n$ matrices in $R_{\bullet}$
and the bottom horizontal map is the inclusion of the invertible matrices.
Let $B\widehat{GL}(R_{\bullet})$ be the classifying space of the simplicial monoid $%
\widehat{GL}(R_{\bullet})=\underset{n}{\underrightarrow{\lim}} \widehat{GL}_{n}(R_{\bullet})$ and then apply Quillen's plus
construction to define
\begin{equation*}
K(R_{\bullet}):=\Z\times B\widehat{GL}(R_{\bullet})^{+}
\end{equation*}
Regard the simplicial monoid $\widehat{GL}(R_{\bullet})$ as a
simplicial category with one object in every simplicial degree and consider the simplicial contravariant
functor $\widehat{GL}(R_{\bullet})\rightarrow \widehat{GL}(R_{\bullet})$  given by $$\left(
\widehat{GL}_{n}(R_{\bullet})\right) _{p}\rightarrow \left( \widehat{GL}_{n}(R_{\bullet})\right)
_{p}\subset M_{n}(R_{p}):(a_{ij})\mapsto (\overline{a_{ji}})$$ where the
conjugation $\overline{a}$ of $a\in R_{p}$ is induced by linearly extending
the inverse map of the group $(G(X))_{p}$. Then it follows from Lemma \ref%
{simplicialfunctor} that this induces an involution on the classifying space
$B\widehat{GL}(R_{\bullet})$. Applying the plus construction, this gives the Burghelea-Fiedorowicz's involution
\beq
\tau_{BF}:K(R_{\bullet})\to K(R_{\bullet})
\eeq
and hence induces the involution $\tau_{BF\ast}$ on $K_{\ast}(R_{\bullet})=\pi_{\ast}K(R_{\bullet})$.

Since the constant map $X\rightarrow\ast$ induces a retraction $K(\mathbb{Z}[G(X)])\rightarrow K(\mathbb{Z})$, then the inclusion $\ast\rightarrow X$ induces a monomorphism $K_{\ast}(\mathbb{Z}[\ast])\rightarrow K_{\ast}(\mathbb{Z}[G(X)])$ which commutes with the involution $\tau_{BF\ast}$. This induces an involution on $\widetilde{K}_{\ast}(\mathbb{Z}[G(X)])=K_{\ast}(\mathbb{Z}[G(X)])/K_{\ast}(\mathbb{Z}[\ast])$ which is also denoted by $\tau_{BF\ast}$.

\subsubsection{Vogell's involution}\label{Section:VInv}
In order to define the Vogell's involution $\tau_V$ on $K(R_{\bullet})$, we recall the other model for the $K$-theory of a simplicial ring which appears already in \cite[p.393]{Waldhausen100} as follows.
Consider the category $\mathcal{M}(R_{\bullet})$ of simplicial (right) modules over $R_{\bullet}$
and their $R_{\bullet}$-linear maps which are weak homotopy equivalences. Given two simplicial modules $A$ and $B$ in $\mathcal{M}(R_{\bullet})$%
, we say $B$\textit{\ is obtained from }$A$\textit{\ by attaching of an }$n$%
\textit{-cell} if there is a pushout diagram
$$\xymatrix{R_{\bullet}[\partial \Delta ^{n}]\ar[r]\ar[d]& R_{\bullet}[\Delta ^{n}]\ar[d]\\
A\ar[r]& B}$$
where $R_{\bullet}[Y]$ denotes the simplicial $R_{\bullet}-$module generated by the simplicial set $Y$, namely, for
each $n$, $\left( R_{\bullet}[Y]\right) _{n}$ is the free $R_{n}-$module generated by $%
Y_{n}$. Define $\mathcal{C}$ to be the full subcategory of the modules which are
obtainable from the zero module by attaching of finitely many cells. Define
$$A^n_k:= R_{\bullet}[\vee _{k}S^{n}]/R_{\bullet}[\ast ]$$  where $S^{n}=\Delta ^{n}/\partial
\Delta ^{n}$ and $\vee _{k}S^{n}$ is the one point union of $k$-copies of $%
S^{n}$ at the point $\ast $. Let $\mathcal{C}_{A_{k}^{n}}$ denote the connected
component of $\mathcal{C}$ containing $A_{k}^{n}$. Taking direct limits with respect to
the functors $\mathcal{C}_{A_{k}^{n}}\rightarrow \mathcal{C}_{A_{k}^{n+1}}$ and $%
\mathcal{C}_{A_{k}^{n}}\rightarrow \mathcal{C}_{A_{k+1}^{n}}$ induced by the tensor product $%
\otimes_{\mathbb{Z}}\mathbb{Z}\lbrack S^{1}]/\mathbb{Z}\lbrack \ast ]$ and the natural inclusion $\vee _{k}S^{n}\rightarrow \vee _{k+1}S^{n}$,
one can define
\begin{equation*}
K^{\prime }(R_{\bullet}):=\mathbb{Z}\times\left\vert \underset{n,k}{\underrightarrow{\lim}}\,\mathcal{C}_{A_{k}^{n}}\right\vert
^{+}
\end{equation*}%
which is homotopy equivalent to $K(R_{\bullet})=\mathbb{Z}\times B\widehat{GL}(R_{\bullet})^{+}$ \cite[p.396]{Waldhausen100}.

In order to define the desired involution on $K'(R_{\bullet})$, we enlarge the category
$\mathcal{C}$ to one that includes duality data (as pointed out in \cite[Remark p.306]{Vogelllong}). First note that $A_1^n=R_{\bullet}[S^{n}]/R_{\bullet}[\ast ]$ can be regarded as a right $R_{\bullet}\otimes_{\mathbb{Z}} R_{\bullet}$--module by $$r\cdot
(s\otimes t)=\overline{t}rs$$ for $r\in R_{k}$ and $s\otimes t\in
R_{k}\otimes R_{k}$, where $\overline{t}\in R_k$ is obtained from $t$ by taking inverses in the group $G(X)_k$.
%Vogell also mentions that $K^{\prime }(R_{\bullet})$ can be reconstructed from a lager category
%of $R_{\bullet}$-modules by including duality data which leads to the desired involution. We now describe this construction in detail.
Let $A_{R_{\bullet}}$ denote
$A\otimes_{R_{\bullet}}\pi_0 R_{\bullet}$.
%\beq
%h_*(A)=\pi_* (A\otimes_{R_{\bullet}}\pi_0 R_{\bullet})
%\eeq
Hence if $A$ and $A^{\prime }$ are simplicial modules in the category $\mathcal{C}$ then any $R_{\bullet}\otimes_{\mathbb{Z}}
R_{\bullet}$--map $\omega :A\otimes_{\mathbb{Z}} A^{\prime }\longrightarrow A_1^n$ induces naturally a bilinear map
\begin{equation}\label{eq:bilinear map}
\pi_q(A_{R_{\bullet}})\times\pi_{n-q}(A'_{R_{\bullet}})\to\pi_n(A_{R_{\bullet}}\otimes_{\mathbb{Z}} A'_{R_{\bullet}})\to\pi_n(A^n_{1R_{\bullet}})\cong\pi_0R_{\bullet}
\end{equation}
%\begin{equation}\label{eq:bilinear map}
%\pi _{q}(A\otimes _{R_{\bullet}}\pi_{0}R_{\bullet})\times \pi _{n-q}(A^{\prime }\otimes_{R_{\bullet}}\pi_{0}R_{\bullet})\rightarrow \pi _{n}((A\otimes_{R_{\bullet}}\pi_{0}R_{\bullet})\otimes_{\mathbb{Z}}(A^{\prime }\otimes _{R_{\bullet}}\pi_{0}R_{\bullet}))\rightarrow \pi _{n}(\pi_{0}R_{\bullet}[ S^{n}]/\pi_{0}R_{\bullet}[ \ast ])\cong\pi_{0}R_{\bullet}
%\end{equation}
for all $0\leq q \leq n$. Note that $\pi_{\ast}(A_{k}^{n}\otimes_{R_{\bullet}}\pi_{0}R_{\bullet})=H_{\ast}((\vee _{k}S^{n})\wedge G_{+})\times^{G}E, \ast\times^{G}E;\pi_{0}R_{\bullet})$ where $G=G(X)$ and $E$ is the universal $G$--bundle, c.f. \cite[p.285]{Vogelllong} and \cite[p.171]{Vogellshort}.

%(For example $\pi_{\ast}(A_{k}^{n}\otimes_{R_{\bullet}}\pi_{0}R_{\bullet})=H_{\ast}((\vee _{k}S^{n})\wedge G_{+})\times^{G}E, \ast\times^{G}E;\pi_{0}R_{\bullet})$ where $G=G(X)$ and $E$ is the universal $G$--bundle, c.f. \cite[p.285]{Vogelllong} and \cite[p.171]{Vogellshort}.)
\begin{definition}\
\begin{enumerate}
\item An $R_{\bullet}\otimes_{\mathbb{Z}}
R_{\bullet}$-map $\omega :A\otimes_{\mathbb{Z}} A^{\prime }\longrightarrow A_1^n$ is called a linear $n$--duality map if the bilinear map (\ref{eq:bilinear map}) induces an isomorphism
$$\pi_q(A_{R_{\bullet}})\cong \Hom_{\pi_{0}R_{\bullet}}(\pi _{n-q}(A_{R_{\bullet}}^{\prime }),\pi_{0}R_{\bullet})$$
for all $0\leq q\leq n$.
\item Let $G=G(X)$. The map $$\mu _{n,m}:A_{k}^{n}\otimes_{\mathbb{Z}} A_{k}^{m}\rightarrow A_1^{n+m}$$ is defined as the $\Z$-linearization of the map
 $$(\vee _{k}S^{n})\wedge G_{+}\wedge
(\vee _{k}S^{m})\wedge G_{+}\longrightarrow S^{n+m}\wedge G_{+}$$
induced by the smash product $S^{n}\wedge S^{m}\longrightarrow
S^{n+m}$ and the map $G\times G\rightarrow G$, $(g,h)\mapsto h^{-1}g$.
\item Given a linear $n$-duality map $\omega :A\otimes_{\mathbb{Z}} A^{\prime }\to A_1^n$, let
$\overline{\omega }$ be the
composition $$A^{\prime }\otimes_{\mathbb{Z}} A\cong A\otimes_{\mathbb{Z}} A^{\prime }\overset{\omega }{%
\rightarrow }A_1^n\rightarrow A_1^n$$
where the last map is given by linearly extending the map
\begin{equation*}
\begin{array}{ccccccc}
S^{n}\wedge G_{+} & \approx & \wedge ^{n}S^{1}\wedge G_{+} & \longrightarrow &
\wedge ^{n}S^{1}\wedge G_{+} & \approx & S^{n}\wedge G_{+} \\
&  & \left[ x_{1},x_{2},\cdots ,x_{n},g\right] & \mapsto & \left[
x_{n},x_{n-1},\cdots ,x_{1},g^{-1}\right] &  &
\end{array}%
\end{equation*}
to a self-map of $A_1^n$.

\end{enumerate}
\end{definition}
Note that the maps $\mu_{n,m}$ and $\overline{\omega }$ defined above are linear duality maps by \cite[p.285, Example 1.7]{Vogelllong} and \cite[p.301]{Vogelllong}.

Let $\mathcal{D}\mathcal{C}^{n}$ be the category in which an object is a
triple $\left( A,A^{\prime },\omega \right) $ where $A$ and $A^{\prime }$
are simplicial modules in the category $\mathcal{C}$ and $\omega $ is a linear $n$-duality map. A morphism in $\mathcal{D}\mathcal{C}^{n}$ is a pair $$\left( f,f^{\prime
}\right) :\left( A,A^{\prime },\omega \right) \rightarrow \left( B,B^{\prime
},\eta \right) $$ where $f:A\rightarrow B$ and $f^{\prime }:B^{\prime
}\rightarrow A^{\prime }$ are morphisms in $\mathcal{C}$ such that the following
diagram commutes:
$$\xymatrix{A\otimes_{\mathbb{Z}} B^{\prime }\ar[r]^{f\otimes_{\mathbb{Z}} id}\ar[d]_{id\otimes_{\mathbb{Z}} f^{\prime }}& B\otimes_{\mathbb{Z}} B^{\prime }\ar[d]_{\eta}\\
A\otimes_{\mathbb{Z}} A^{\prime }\ar[r]^-\omega& A_1^n}$$
Define%
\begin{equation*}
\mathcal{D}K(R_{\bullet}):=\mathbb{Z}\times\left\vert \underset{n,m,k}{\underrightarrow{\lim} }\mathcal{D}%
\mathcal{C}_{(A_{k}^{n},A_{k}^{m},\mu_{n,m} )}^{n+m}\right\vert ^{+}
\end{equation*}%
where $\mathcal{D}\mathcal{C}_{(A_{k}^{n},A_{k}^{m},\mu_{n,m} )}^{n+m}$ is the connected
component of $\mathcal{D}\mathcal{C}^{n+m}$ containing the triple $%
(A_{k}^{n},A_{k}^{m},\mu_{n,m} )$ and the limit here is taken with respect to the
functors induced by the tensor product $\otimes_{\mathbb{Z}}\mathbb{Z}\lbrack S^{1}]/\mathbb{Z}\lbrack \ast ]$ and the natural inclusion $\vee _{k}S^{n}\rightarrow \vee _{k+1}S^{n}$.
By Lemma \ref{simplicialfunctor}, an involution $$\tau_{V}:\mathcal{D}K(R_{\bullet})\to\mathcal{D}K(R_{\bullet})$$ can be induced by the anti-involution (contravariant functor) $%
\mathcal{D}\mathcal{C}^{n+m}\rightarrow \mathcal{D}\mathcal{C}^{n+m}$ which sends an object $%
\left( A,A^{\prime },\omega \right) $ to $\left( A^{\prime },A,\overline{%
\omega }\right) $ and sends a morphism $\left( f,f^{\prime }\right) $ to $%
\left( f^{\prime },f\right) $. Analogously to the proof of \cite[Corollary 1.16]{Vogelllong}, one can show that there is a homotopy equivalence
\begin{equation}\label{eq:dual-nondual}
\mathcal{D}K(R_{\bullet})\rightarrow K^{\prime }(R_{\bullet})
\end{equation}
 induced by $\mathcal{%
D}\mathcal{C}_{(A_{k}^{n},A_{k}^{m},\mu_{n,m})}^{n+m}\longrightarrow \mathcal{C}_{A_{k}^{n}}$ which
maps an object $\left( A,A^{\prime },\omega \right) $ to $A$ and maps a
morphism $\left( f,f^{\prime }\right) $ to $f$. This yields an involution on $K'(R_{\bullet})$ and hence on $K(R_{\bullet})$, as $K(R_{\bullet})$ is homotopy equivalent to $K'(R_{\bullet})$ \cite[Corollary p.396]{Waldhausen100}.

Finally note that the constant map $X\rightarrow \ast$ induces a retraction $\mathcal{D}K(\mathbb{Z}[G(X)])\rightarrow \mathcal{D}K(\mathbb{Z}[\ast])$ and the inclusion $\mathcal{D}K(\mathbb{Z}[\ast])\rightarrow\mathcal{D}K(\mathbb{Z}[G(X)])$ commutes with the involution $\tau_{V}$. This induces the involution $$\tau_{V\ast}:\pi_{\ast}(\mathcal{D}K(\mathbb{Z}[G(X)]),\mathcal{D}K(\mathbb{Z}[\ast]))\to\pi_{\ast}(\mathcal{D}K(\mathbb{Z}[G(X)]),\mathcal{D}K(\mathbb{Z}[\ast]))$$.
\begin{lemma}\label{Lemma:VInv}
Let $A(\left\vert X \right\vert)$ denote the Waldhausen's $K$-theory of the geometric realization of $X$. Then there is a map $A(\left\vert X \right\vert)\rightarrow\mathcal{D}K(\mathbb{Z}[G(X)])$ inducing an isomorphism $$\pi_{i}(A(\left\vert X \right\vert),A(\ast))\otimes \mathbb{Q}\rightarrow \pi_{i}(\mathcal{D}K(\mathbb{Z}[G(X)]),\mathcal{D}K(\mathbb{Z}(\ast)))\otimes \mathbb{Q}$$for all $i$  such that the following diagram commutes:
$$\xymatrix{\pi_{i}(A(\left\vert X \right\vert),A(\ast))\otimes \mathbb{Q}\ar[r]\ar[d]_{\tau_{\varepsilon\ast}} & \pi_{i}(\mathcal{D}K(\mathbb{Z}[G(X)]),\mathcal{D}K(\mathbb{Z}(\ast)))\otimes \mathbb{Q}\ar[d]^{\tau_{V\ast}}\\
\pi_{i}(A(\left\vert X \right\vert),A(\ast))\otimes \mathbb{Q}\ar[r] & \pi_{i}(\mathcal{D}K(\mathbb{Z}[G(X)]),\mathcal{D}K(\mathbb{Z}(\ast)))\otimes \mathbb{Q}}$$
\end{lemma}
Before we prove the lemma we need to recall yet another model for $A(\left\vert X \right\vert)$.  Let $G=G(X)$ be the Kan loop group of $X$. Consider the category
$h\mathcal{U}_{hf}(G)$ whose objects are pointed simplicial sets with a right free $G$-action (in the pointed sense) and which are finitely generated over $G$. Morphisms are $G$-maps which are weak homotopy equivalences. Restrict $h\mathcal{U}_{hf}(G)$ to the full subcategory $h\mathcal{U}_{k}^{n}(G)$ whose objects are homotopy equivalent to a wedge of $k$ spheres of dimension $n$. This category is an approximation to $A(\left\vert X \right\vert)$ \cite[Corollary and/or Definition, p.42]{WaldhausenI1978}, i.e. there is a homotopy equivalence
\begin{equation}\label{eqn: equiv-models}
A(\left\vert X \right\vert)\simeq\mathbb{Z}\times |\underset{k,n}{\underrightarrow{\lim}}\; h\mathcal{U}_{k}^{n}(G)|^{+}
\end{equation}
From this point of view, the linearization map $A(\left\vert X \right\vert)\rightarrow K'(\mathbb{Z}[G(X)])$ can be easily described: it sends a free pointed simplicial $G$-set to the free simplicial
$\mathbb{Z}[G]$-module generated by the nonbasepoint elements\cite{WaldhausenI1978}.

One can enlarge this category to a category
$h\mathcal{DU}_{k}^{l,m}(G)$ by adding duality data. We sketch the definition  as follows and refer the readers to \cite{Vogelllong} for details. An $(l+m)$-dual of an object $N$ in $h\mathcal{U}_{k}^{l}(G)$ is an object $N'$
in $h\mathcal{U}_{k}^{m}(G)$ together with a pointed right $(G\times G)-$map $u:N\wedge N'\to S^{l+m}\wedge G_+$ such that it induces an isomorphism of $\Z[\pi_0G]$-modules (via slant product)
\beq
H_q^G(N')\to H_G^{l+m-q}(N)
\eeq
for $0\leq q\leq n$. Here $u$ is called an $(l+m)-$duality map. An object in  $h\mathcal{DU}_{k}^{l,m}(G)$ should then be a triple $(N,N',u)$ where $N$ (resp. $N'$) is an object in $h\calU^l_k(G)$ (resp. $h\calU^m_k(G)$) and $u:N\wedge N'\to S^{l+m}\wedge G_+$ is an $(l+m)-$duality map. Morphisms are defined as we did above for the category $\mathcal{DC}^n$.

Now define $\mathcal{D}K(\left\vert X \right\vert):= \mathbb{Z}\times |\underset{k,l,m}{\underrightarrow{\lim}} h\mathcal{DU}_{k}^{l,m}(G)|^{+}$.  Similarly to (\ref{eq:dual-nondual}) there is a homotopy equivalence \cite[p.293, Remark]{Vogelllong}
\begin{equation}\label{eqn:forget-dual}
\mathcal{D}K(\left\vert X \right\vert)\xrightarrow{\simeq}\Z\times|\underset{k,n}{\underrightarrow{\lim}} h\mathcal{U}_{k}^{n}(G)|^{+}\simeq A(\left\vert X \right\vert)
\end{equation}
Furthermore, just as
$\mathcal{D}K(\Z[G(X)])$, the space $\mathcal{D}K(\left\vert X\right\vert)$ acquires a natural involution, and it follows from \cite[Proposition 1.14]{Vogelllong} that the equivalence (\ref{eqn:forget-dual}) is equivariant if the target is given the involution $\tau_{\varepsilon}$. We will refer to these two involutions on $A(\left\vert X \right\vert)$ with the same symbol $\tau_{\varepsilon}$.
\begin{proof}[Proof of Lemma \ref{Lemma:VInv}]
Note that the definitions have been made so that the linearization map $A(\left\vert X \right\vert)\to K'(\Z[G(X)])$ extends to a (linearization) map $\mathcal{D}K(\left\vert X \right\vert)\to\mathcal{D}K(\Z[G(X)])$ equivariant with respect to the involutions. In other words there is a commutative diagram
\begin{equation}\label{diagram:linearization}
\xymatrix{\mathcal{D}K(|X|)\ar[r]\ar[d]_{\simeq} & \mathcal{D}K(\mathbb{Z}[G(X)])\ar[d]^{\simeq}\\
A(|X|)\ar[r]& K'(\mathbb{Z}[G(X)])}
\end{equation}
where the right vertical map is the map (\ref{eq:dual-nondual}). The left vertical is the map (\ref{eqn:forget-dual}).
Since the bottom map is a rational homotopy equivalence \cite[Proposition 2.2]{WaldhausenI1978}, so is the top map. Moreover the diagram restricts to a commutative diagram when $X=\ast$. Then a homotopy inverse of the left vertical map in the diagram (\ref{diagram:linearization}) composed with the top horizontal map gives the desired map $A(|X|)\rightarrow \mathcal{D}K(\mathbb{Z}[G(X)])$ and this completes the proof.
\end{proof}

\subsubsection{The involutions $\tau_{BF} $ and $\tau_{V}$ coincide}\label{Section:BFInvVInv}

We now show that the two involutions $\tau_{BF} $ and $\tau_{V}$ coincide.

\begin{theorem}
\label{Invcoincide} Let $X$ be a simplicial set. There is a homotopy equivalence $$\eta_{X}: K(\mathbb{Z}[G(X)])\rightarrow \mathcal{D}K(\mathbb{Z}[G(X)])$$ such that the first diagram below commutes up to homotopy and the second diagram commutes
\begin{equation}\label{diagram:eta}
\xymatrix{K(\mathbb{Z}[G(X)])\ar[r]^{\eta_{X}}\ar[d]_{\tau_{BF}}&\mathcal{D}K(\mathbb{Z}[G(X)])\ar[d]_{\tau_{V}}\\
K(\mathbb{Z}[G(X)])\ar[r]^{\eta_{X}}&\mathcal{D}K(\mathbb{Z}[G(X)])}\ \ \ \
\xymatrix{K(\mathbb{Z}[\ast])\ar[r]^{\eta_{\ast}}\ar[d]&\mathcal{D}K(\mathbb{Z}[\ast])\ar[d]\\K(\mathbb{Z}[G(X)])\ar[r]^{\eta_{X}}&\mathcal{D}K(\mathbb{Z}[G(X)])}
\end{equation}

\end{theorem}
\begin{proof}
Consider the following commutative diagram
\begin{equation}\label{diagram:KDK}
\xymatrix{&\mathcal{D}\mathcal{C}_{\bullet (A_{k}^{0},A_{k}^{0},\mu _{0,0})}^{0}\ar[r]^-{\mathcal{D}\Sigma}\ar[2,0]&\mathcal{D}\mathcal{C}_{\bullet (A_{k}^{n},A_{k}^{m},\mu _{n,m})}^{n+m}\ar[2,0]&\ar[l]_-{\mathcal{D}\Upsilon}\ar[2,0]\mathcal{D}\mathcal{C}_{(A_{k}^{n},A_{k}^{m},\mu _{n,m})}^{n+m}\\
\widehat{GL}_{k}(R_{\bullet})\ar[ur]^\Phi\ar[dr]_\Psi&&&\\
&\mathcal{C}_{\bullet A_{k}^{0}}\ar[r]^\Sigma&\mathcal{C}_{\bullet
A_{k}^{n}} &\ar[l]_{\Upsilon}\mathcal{C}_{A_{k}^{n}}}
\end{equation}

where the notations are defined below.
\begin{enumerate}
\item  $\mathcal{C}_{\bullet}$ is the simplicial category
$[m]\to \mathcal{C}_m$, where the objects of $\mathcal{C}_m$ are the same as those of $\mathcal{C}$ and the morphisms in $\mathcal{C}_m$ are $m$-parameter families of morphisms in $\mathcal{C}$ (see [Wal85, p.396]). Similarly
$\mathcal{D}\mathcal{C}_{\bullet}^{n}$ is a simplicial category defined as follows: for each $p$, the objects in the category $%
\mathcal{D}\mathcal{C}_{p}^{n}$ are the same with those in $\mathcal{D}\mathcal{C}^{n}$, and a
morphism $\left( A,A^{\prime },\omega \right) \rightarrow \left( B,B^{\prime
},\eta \right) $ in $\mathcal{D}\mathcal{C}_{p}^{n}$ consists of a pair of morphisms $%
F:A\otimes_{\mathbb{Z}}\mathbb{Z}\left[ \Delta ^{p}\right] \rightarrow B$ and $F^{\prime }:B^{\prime }\otimes_{\mathbb{Z}}\mathbb{Z}\left[ \Delta ^{p}\right] \rightarrow A^{\prime }$ in the category $\mathcal{C}$ such
that the following diagram commutes:
$$\xymatrix{A\otimes_{\mathbb{Z}}\mathbb{Z}\left[\Delta ^{p}\right] \otimes_{\mathbb{Z}} B^{\prime }\ar[r]^{\approx}\ar[d]_{F\otimes_{\mathbb{Z}} id}&A\otimes_{\mathbb{Z}} B^{\prime
}\otimes_{\mathbb{Z}}\mathbb{Z}\left[\Delta ^{p}\right]\ar[r]^-{id\otimes_{\mathbb{Z}} F^{\prime }} & A\otimes_{\mathbb{Z}} A^{\prime }\ar[d]_\omega \\
B\otimes_{\mathbb{Z}} B^{\prime }\ar[0,2]^\eta&& R_{\bullet}[S^{n}]/R_{\bullet}[\ast ]}$$
The canonical anti-involution in $\mathcal{D}\mathcal{C}_{\bullet}^{n}$ is defined
similarly as that in the category $\mathcal{D}\mathcal{C}^{n}$. Denote by $\mathcal{D}%
\mathcal{C}_{\bullet (A_{k}^{n},A_{k}^{m},\mu _{n,m})}^{n+m}$ the connected
component of $\mathcal{D}\mathcal{C}_{\bullet}^{n+m}$ containing $(A_{k}^{n},A_{k}^{m},%
\mu _{n,m})$ and regard $\mathcal{D}\mathcal{C}_{(A_{k}^{n},A_{k}^{m},\mu
_{n,m})}^{n+m}$ as a constant simplicial category.
\item\label{point} The horizontal map $\mathcal{D}\Upsilon:
\mathcal{D}\mathcal{C}_{(A_{k}^{n},A_{k}^{m},\mu _{n,m})}^{n+m}\rightarrow \mathcal{D}%
\mathcal{C}_{\bullet (A_{k}^{n},A_{k}^{m},\mu _{n,m})}^{n+m}$ is induced by associating $%
f:A\rightarrow B$ to $f\otimes_{\mathbb{Z}} c:A\otimes_{\mathbb{Z}}\mathbb{Z}\left[ \Delta ^{p}\right] \rightarrow B\otimes_{\mathbb{Z}}\mathbb{Z}$, where $c:\mathbb{Z}\left[ \Delta ^{p}\right] \rightarrow\mathbb{Z}$ is the $\mathbb{Z}-$linear map induced by the constant map $\Delta ^{p}\rightarrow \ast$. It is a homotopy equivalence by the same argument as in [Wal85, Proposition 2.2.5].
 $\mathcal{C}_{\bullet A_{k}^{n}}$ denotes the connected
component of $\mathcal{C}_{\bullet}$ containing $A_{k}^{n}$ and the map $\Upsilon: \mathcal{C}_{A_{k}^{n}}\rightarrow $ $\mathcal{C}_{\bullet
A_{k}^{n}}$ is defined similarly.
  \item The three vertical maps are all induced by the projections onto the first
factors.
  \item The horizontal maps $\mathcal{D}\Sigma: \mathcal{D}\mathcal{C}_{\bullet (A_{k}^{0},A_{k}^{0},\mu
_{0,0})}^{0}\rightarrow \mathcal{D}\mathcal{C}_{\bullet (A_{k}^{n},A_{k}^{m},\mu
_{n,m})}^{n+m}$ and $\Sigma :\mathcal{C}_{\bullet A_{k}^{0}}\rightarrow C_{\bullet
A_{k}^{n}}$ are induced from the tensor product $\otimes_{\mathbb{Z}}\mathbb{Z}\lbrack S^{n}]/\mathbb{Z}\lbrack \ast ]$ and $\otimes_{\mathbb{Z}}\mathbb{Z}\lbrack S^{m}]/\mathbb{Z}\lbrack \ast ]$, respectively.
 \item Each $M\in \left( \widehat{GL}_{k}(R_{\bullet})\right) _{p}\subset M_{k}(R_{p})$
determines a unique morphism $f_{M}:A_{k}^{0}\otimes_{\mathbb{Z}}\mathbb{Z}\left[ \Delta ^{p}\right] \rightarrow A_{k}^{0}$ in the category $\mathcal{C}$ such
that $\left( f_{M}\right) _{p}(v\otimes \iota _{p})=Mv$ for any column
vector $v\in R_{p}^{k}=(A_{k}^{0})_{p}$ and the non-degenerate $p-$simplex $\iota _{p}$ of $\Delta ^{p}$. This induces the natural map $\Psi :%
\widehat{GL}_{k}(R_{\bullet})\rightarrow \mathcal{C}_{\bullet A_{k}^{0}}$ when considering $\widehat{GL}_{k}(R_{\bullet})$ as a simplicial category with one object in each simplicial degree. Similarly, $\Phi :%
\widehat{GL}_{k}(R_{\bullet})\rightarrow \mathcal{D}\mathcal{C}_{\bullet (A_{k}^{0},A_{k}^{0},\mu
_{0,0})}^{0}$ is the natural map which maps each $M\in \left( \widehat{GL}%
_{k}(R_{\bullet})\right) _{p}$ to the morphism $\left( f_{M},f_{\overline{M}%
^{T}}\right) :(A_{k}^{0},A_{k}^{0},\mu _{0,0})\rightarrow
(A_{k}^{0},A_{k}^{0},\mu _{0,0})$.
\end{enumerate}
Since the simplicial maps $\Phi$, $\mathcal{D}\Sigma$ and $\mathcal{D}\Upsilon$ in diagram (\ref{diagram:KDK}) preserve the
anti-involutions, and since $\mathcal{D}\Upsilon$ is a homotopy equivalence, then after passing to limits, the composition of $\mathcal{D}\Sigma\circ\Phi$ with a homotopy inverse of $\mathcal{D}\Upsilon$ gives rise to a map $$\eta_{X}: K(\mathbb{Z}[G(X)])\rightarrow \mathcal{D}K(\mathbb{Z}[G(X)])$$ which fits into the commutative diagrams (\ref{diagram:eta}). To prove the theorem it suffices to prove that $\eta_{X}$ is a homotopy equivalence. In fact, since the vertical map on the right in the diagram (\ref{diagram:KDK}) is a homotopy equivalence when passing to limits, and the simplicial maps $\Psi $ and $\Upsilon$ are both homotopy equivalences (c.f. \cite[p.396, Proposition 2.3.5]{Waldhausen100}), by the commutativity of the diagram (\ref{diagram:KDK}) we only need to show that the simplicial map $\Sigma $ is a homotopy equivalence.
For this, consider the following commutative diagram:
$$\xymatrix{\mathcal{C}_{\bullet A_{k}^{0}}\ar[r]^\Sigma & \mathcal{C}_{\bullet A_{k}^{n}}\\
\mathcal{C}_{\bullet}(A_{k}^{0})\ar[r]\ar[u]& \mathcal{C}_{\bullet}(A_{k}^{n})\ar[u]}$$
where $\mathcal{C}_{\bullet}(A_{k}^{n})$ denotes the simplicial monoid of simplicial $R_{\bullet}-$linear self-homotopy
equivalences of $A_{k}^{n}$ (c.f.\cite[I.7]{GJ}), the vertical maps are the
natural inclusions, and the bottom horizontal map is the restriction of the
``suspension" map $\Hom_{R_{\bullet}}(A_{k}^{0},A_{k}^{0})\rightarrow
\Hom_{R_{\bullet}}(A_{k}^{n},A_{k}^{n})$ induced from the tensor product $\mathbb{Z}\lbrack S^{n}]/\mathbb{Z}\lbrack \ast ]\otimes_{\mathbb{Z}}$, where $\Hom_{R_{\bullet}}(A_{k}^{n},A_{k}^{n})$ is the simplicial
monoid of simplicial $R_{\bullet}-$linear self-maps of $A_{k}^{n}$. Since the vertical
maps are homotopy equivalences \cite[p.396, Proposition 2.3.5]{Waldhausen100}, it remains to show
\begin{equation}\label{suspension map}
\Hom_{R_{\bullet}}(A_{k}^{0},A_{k}^{0})\rightarrow
\Hom_{R_{\bullet}}(A_{k}^{n},A_{k}^{n})
\end{equation}
is a homotopy equivalence. To see this, let $A$ be a simplicial abelian group and $\widetilde{\mathbb{Z}}[S^{n}]=\mathbb{Z}\lbrack S^{n}]/\mathbb{Z}\lbrack \ast ]$. As $A_{p}=\Hom_{\mathbb{Z}}(\mathbb{Z}\lbrack \Delta ^{p}],A)$ for each dimension $p$, define a map
\begin{equation}\label{eq:map}
A\xrightarrow{\widetilde{\mathbb{Z}}[S^{1}]\otimes_{\mathbb{Z}} }\Hom_{\mathbb{Z}}(\widetilde{\mathbb{Z}}[S^{1}],\widetilde{\mathbb{Z}}[S^{1}]\otimes A)
\end{equation}
by sending a simplicial $\mathbb{Z}-$linear map $f:\mathbb{Z}\lbrack \Delta ^{p}]\rightarrow A$ to the map $id\otimes f:\widetilde{\mathbb{Z}}[S^{1}]\otimes\mathbb{Z}\lbrack \Delta ^{p}]\rightarrow\widetilde{\mathbb{Z}}[S^{1}]\otimes A$ in each dimension $p$.

\smallskip

\textbf{Claim.} The map (\ref{eq:map}) is a homotopy equivalence.

\smallskip

\noindent Assuming this claim we prove the ``suspension" map (\ref{suspension map}) is a
homotopy equivalence. Since $$\displaystyle \Hom_{R_{\bullet}}(A_{k}^{n},A_{k}^{n})=\bigoplus_{k\times
k}\Hom_{R_{\bullet}}(A_{1}^{n},A_{1}^{n}),$$ it suffices to prove that the ``suspension" map (\ref{suspension map}) is a homotopy equivalence for $k=1$. Firstly, it is true for $n=1$
by using the claim above. Now assume it is true for $n=l-1$. Then the fact that $%
\Hom_{R_{\bullet}}(A_{1}^{0},A_{1}^{0})\rightarrow \Hom_{R_{\bullet}}(A_{1}^{l},A_{1}^{l})$ is a
homotopy equivalence just follows from that it is the composition of the
homotopy equivalences

\begin{equation*}
\Hom_{R_{\bullet}}(A_{1}^{0},A_{1}^{0})\xrightarrow{\widetilde{\mathbb{Z}}[S^{l-1}]\otimes}
\Hom_{R_{\bullet}}(A_{1}^{l-1},A_{1}^{l-1})\to\Hom_{R_{\bullet}}(A_{1}^{l-1},\Hom_{\mathbb{Z}}(\widetilde{\mathbb{Z}}[S^{1}],\widetilde{\mathbb{Z}}[S^{1}]\otimes A_{1}^{l-1}))\cong \Hom_{R_{\bullet}}(A_{1}^{l},A_{1}^{l})
\end{equation*}
where the second map is induced by the map in the claim above when $A=A_{1}^{l-1}$. Then it follows that the ``suspension" map (\ref{suspension map})  is a
homotopy equivalence for all $k\geq 1$.

It remains to prove the claim above. The
proof follows essentially from  \cite[Proposition 2.3.5]{Waldhausen100} and a simplicial version of the proof of \cite[Proposition
4.66]{Hatcher}. Let $$c:\Delta ^{1}\times (A\left[ \Delta ^{1}\right] /A\left[
\Delta ^{0}\right] )\rightarrow A\left[ \Delta ^{1}\right] /A\left[ \Delta
^{0}\right] $$ be the contraction which linearly extends the standard
contraction $\Delta ^{1}\times \Delta ^{1}\rightarrow \Delta ^{1}$, whose
restrictions to $\Delta ^{1}\times 0$ and $0\times \Delta ^{1}$ are the
projections of $\Delta ^{1}$ onto the vertex $0$ and whose restriction to $%
\Delta ^{1}\times 1$ is the identity map of $\Delta ^{1}$. Let $P(A\left[
\Delta ^{1}\right] /A\left[\partial\Delta ^{1}\right] )$ (resp. $\Omega (A\left[
\Delta ^{1}\right] /A\left[\partial \Delta ^{1}\right] )$) denote the path space
(resp. the loop space) of $A\left[ \Delta ^{1}\right] /A\left[\partial \Delta ^{1}%
\right] $ (see \cite[p.30]{GJ}). Then the contraction $c$ induces naturally
a map $A\left[ \Delta ^{1}\right] /A\left[ \Delta ^{0}\right] \rightarrow P(A%
\left[ \Delta ^{1}\right] /A\left[\partial \Delta ^{1}\right] )$ which fits into the following commutative diagram
$$\xymatrix{A\ar[0,2]\ar[d]_{\widetilde{\mathbb{Z}}[S^{1}]\otimes}&& A\left[ \Delta ^{1}\right] /A\left[\Delta ^{0}\right]\ar[d]\ar[r]& A\left[ \Delta ^{1}\right] /A\left[\partial \Delta ^{1}\right]\ar[d]^{id} \\
\Hom_{\mathbb{Z}}(\widetilde{\mathbb{Z}}[S^{1}],\widetilde{\mathbb{Z}}[S^{1}]\otimes A)\ar[r]^-{\approx}& \Omega (A\left[ \Delta ^{1}\right] /A\left[\partial \Delta
^{1}\right] )\ar[r]& P(A\left[ \Delta ^{1}\right] /A\left[\partial \Delta
^{1}\right] )\ar[r] & A\left[ \Delta ^{1}\right] /A\left[\partial \Delta ^{1}\right]}$$
where the horizontal maps are the natural fibrations. Since $A\left[ \Delta
^{1}\right] /A\left[ \Delta ^{0}\right] $ and $P(A\left[ \Delta ^{1}\right]
/A\left[\partial \Delta ^{1}\right] )$ are contractible, the five lemma implies that
the left vertical map is a weak homotopy equivalence. As the simplicial abelian groups $A$ and $\Hom_{\mathbb{Z}}(\widetilde{\mathbb{Z}}[S^{1}],\widetilde{\mathbb{Z}}[S^{1}]\otimes A)$ are Kan complexes, then the claim follows. This completes the proof of the theorem.
\end{proof}

\begin{proof}[\textbf{Proof of Theorem \ref{Theorem:AinvKinv}}]Theorem \ref{Invcoincide} implies that the homotopy equivalence $\eta_{X}: K(\mathbb{Z}[G(X)])\rightarrow\mathcal{D}K(\mathbb{Z}[G(X)])$ induces isomorphisms $$\eta_{X\ast}:\widetilde{K}_{i}(\mathbb{Z}[G(X)])\overset{\cong}{\rightarrow} \pi_{i}(\mathcal{D}K(\mathbb{Z}[G(X)]),\mathcal{D}K(\mathbb{Z}))$$ such that the following diagram commutes
$$\xymatrix{\widetilde{K}_{i}(\mathbb{Z}[G(X)])\ar[r]^-{\eta_{X\ast}}\ar[d]_{\tau_{BF\ast}}&\pi_{i}(\mathcal{D}K(\mathbb{Z}[G(X)]),\mathcal{D}K(\mathbb{Z}))\ar[d]^{\tau_{V\ast}}\\
\widetilde{K}_{i}(\mathbb{Z}[G(X)])\ar[r]^-{\eta_{X\ast}}&\pi_{i}(\mathcal{D}K(\mathbb{Z}[G(X)]),\mathcal{D}K(\mathbb{Z}))}$$for all $i$.
This, together with Lemma \ref{Lemma:VInv} completes the proof.
\end{proof}

\subsection{The relation between the involutions on $K_{\ast }(\mathbb{Z}[G(X)])$ and $H_{i}^{S^{1}}(L\left\vert X\right\vert )$}

Let $X$ be a simply-connected simplicial set and let $HD_{i}(X)$ denote
the $i$-th dihedral homology for the simplicial Hermitian ring $\mathbb{Z}[G(X)]$ (see for example \cite[p.193]{Lodder96}). An anti-equivariant isomorphism $\widetilde{K}_{i+1}(\mathbb{Z}[G(X)])\otimes\mathbb{Q}\cong H_{i}^{S^{1}}(L\left\vert X\right\vert ,\ast ;\mathbb{Q})$ with respect to the involution $\tau_{BF\ast}$ and the geometric involution $T_{\ast}$, can be obtained by recalling the following results: On one hand, applying
the results of \cite[p.195]{Lodder96} directly to the Hermitian simplicial ring homomorphism $\mathbb{Z}[G(X)]\rightarrow\mathbb{Z}$ induced by $G(X)\rightarrow \ast$, one can see there is an isomorphism
\begin{equation*}
\Inv^{-}\widetilde{K}_{i+1}(\mathbb{Z}[G(X)])\otimes\mathbb{Q}\cong \widetilde{HD}_{i}(X)\otimes\mathbb{Q}:=HD_{i}(X)\otimes\mathbb{Q}/HD_{i}(\ast )\otimes\mathbb{Q}.
\end{equation*}
On the other hand, Dunn \cite{Dunn} proves that $HD_{i}(X)\otimes\mathbb{Q}$ is isomorphic to $H_{i}^{O(2)}(L\left\vert X\right\vert ;\mathbb{Q})$, where $O(2)$ acts on $L|X|$ by reparametrization. A modification of the proof of \cite[3.3.3 Theorem]{Lodder90} can
also show this.
These facts, together with the isomorphism (obtained by a transfer map argument)
\begin{equation*}
H_{i}^{O(2)}(L\left\vert X\right\vert ;\mathbb{Q})\cong \Inv^{+}H_{i}^{S^{1}}(L\left\vert X\right\vert ;\mathbb{Q}),
\end{equation*}
imply that $$\Inv^{-}\widetilde{K}_{i+1}(\mathbb{Z}[G(X)])\otimes\mathbb{Q}\cong \Inv^{+}H_{i}^{S^{1}}(L\left\vert X\right\vert ,\ast ;\mathbb{Q}).$$ Note that $\dim\widetilde{K}_{i}(\mathbb{Z}[G(X)])\otimes\mathbb{Q}$ is finite for each $i$ when the geometric realization $|X|$ has the homotopy type of a simply-connected compact manifold \cite{Dwyer,WaldhausenI1978}, and since $\widetilde{K}_{i+1}(\mathbb{Z}[G(X)])\otimes\mathbb{Q}\cong H_{i}^{S^{1}}(L\left\vert X\right\vert ,\ast ;\mathbb{Q})$ by \cite{B,BF,Goodwillie85,Goodwillie86}, then the following theorem holds.
\begin{theorem}\label{Theorem:KS1}
Let $X$ be a simply-connected simplicial set whose geometric realization $|X|$ has the homotopy type of a compact manifold. Then there is an isomorphism $\xymatrix{\widetilde{K}_{i+1}(\mathbb{Z}[G(X)])\otimes\mathbb{Q}\ar[r]^-{\cong}& H_{i}^{S^{1}}(L\left\vert X\right\vert ,\ast ;\mathbb{Q})}$ for all $i$ such that the following diagram commutes:$$\xymatrix{\widetilde{K}_{i+1}(\mathbb{Z}[G(X)])\otimes\mathbb{Q}\ar[r]^-{\cong}\ar[d]_-{\tau_{BF\ast}}& H_{i}^{S^{1}}(L\left\vert X\right\vert ,\ast ;\mathbb{Q})\ar[d]^{-T_{\ast}}\\\widetilde{K}_{i+1}(\mathbb{Z}[G(X)])\otimes\mathbb{Q}\ar[r]^-{\cong}& H_{i}^{S^{1}}(L\left\vert X\right\vert ,\ast ;\mathbb{Q})}$$

\end{theorem}

\begin{proof}[\textbf{Proof of Theorem \ref{AthS1}}] It follows immediately from  Theorems \ref{Theorem:AinvKinv} and \ref{Theorem:KS1}.
\end{proof}

\begin{remark}Lodder uses two apparently different definitions of dihedral homologies in \cite{Lodder90} and \cite{Lodder96}. However a straightforward argument shows that they are equivalent.
\end{remark}
\section{The involution on the pseudoisotopy space}\label{InvPseudo}
Throughout this section let $M$ be an orientable
compact smooth manifold, possibly with boundary. The goal of this section is to introduce the involution on pseudoisotopy spaces and to prove Proposition \ref{equivariantshortexact}, which relates the involution $\iota^S$ on the homotopy groups of $\mathscr{P}(M)$ to the involution $\tau_{\varepsilon}$ on $A(M)$. The proof is based on Waldhausen's manifold approach to $A(M)$
\cite{Waldhausenmfd} and arguments in \cite{Vogelllong}. We prove Theorem \ref{formula} at the end of this section.

A pseudoisotopy of $M$ is
defined to be a diffeomorphism of $M\times [0,1]$ which is the identity on $%
M\times 0\cup \partial M\times[0,1]$. Then the pseudoisotopy space is the
group of all such diffeomorphisms equipped with the smooth topology, i.e.,
\begin{equation*}
P(M):=\Diff(M\times[0,1], rel\ M\times 0\cup \partial M\times[0,1]).
\end{equation*} By Igusa's stabilization theorem \cite{Igusa88} \cite[Theorem 6.2.2]{Igusa02}, the natural
stabilization map
\begin{equation*}
\Sigma: P(M)\rightarrow P(M\times [0,1])
\end{equation*}
is $k-$connected if $\dim M\geq \max\{2k+7,3k+4\}$. Then the
stabilization
\begin{equation*}
P(M)\rightarrow P(M\times [0,1])\rightarrow P(M\times
[0,1]^{2})\rightarrow\cdots
\end{equation*}
is eventually an isomorphism on homotopy groups. Define the stable
pseudoisotopy space as the direct limit
\begin{equation*}
\mathscr{P}(M):=\underset{m}{\underrightarrow{\lim}}P(M\times [0,1]^{m}).
\end{equation*}
Following \cite[p.296]{Vogelllong}, there is a canonical involution $\iota:
P(M)\rightarrow P(M)$ given by
\begin{equation*}
\iota(f)=(id_{M}\times\ r)\circ f \circ (id_{M}\times\ r)\circ((f|_{M\times
1})^{-1}\times id_{[0,1]})
\end{equation*}
where $r:[0,1]\rightarrow [0,1]$ is the reflection given by $r(t)=1-t$.
Since the stabilization induces a map $\pi_{*}P(M)\rightarrow \pi_{*}P(M\times[0,1])$
which anti-commutes with the canonical involution (c.f. \cite{Hatcher78} or \cite[Proposition 6.2.1, Lemma 6.5.1(2)]{Igusa02}), then define the involution $\iota_{\ast}^{S}$ on $\pi_{*}\mathscr{
P}(M)$ to be the one compatible with the involution $(-1)^{\dim M}\iota_{*}$
on $\pi_{*}P(M)$ (c.f. \cite[p.251]{Igusa02}), namely the following diagram commutes:
$$\xymatrix{\pi_{k}P(M)\ar[r]\ar[d]_{(-1)^{\dim M}\iota_{\ast}}&\pi_{k}\mathscr{P}(M)\ar[d]_{\iota^{S}_{\ast}}\\
\pi_{k}P(M)\ar[r]&\pi_{k}\mathscr{P}(M)} $$

The link between this involution on $\pi_{\ast}\mathscr{P}(M)$ and the involutions on $K$-theory (of spaces) appears through  Waldhausen's manifold approach to $A(M)$ \cite{Waldhausenmfd,WJR} and Vogell's work on involutions\cite{Vogelllong}; which we review now.

Given a compact manifold $M$ with boundary $\partial M$, a partition is
a triple $(M_{0},F,M_{1})$ of manifolds, as shown in Figure \ref{partition},
\begin{figure}[!h]
\vspace*{0cm}
    \begin{center}
    \includegraphics[scale=0.6]{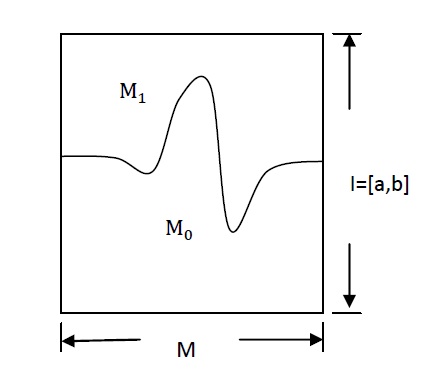}
    \caption{\small A partition $(M_{0},F,M_{1})$ }\label{partition}
    \end{center}
 \end{figure}
where $M_{0}$ is a codimension $0$
submanifold of $M\times I$ with $I=[a,b]$, $M_{1}$ is the closure of the complement of $M_{0}$ and $F$ is the intersection $M_{0}\cap M_{1}$. Furthermore, the frontier $F$ is required to be disjoint
from the bottom $M\times a$ and the top $M\times b$ and to be standard near $\partial M\times I$ in the sense that there exists a neighborhood $W$ of $\partial
M\times I$ satisfying that $W\cap F$ is equal to $W\cap M\times \{t\}$ for some $t\in I$. Let $\calP(M)$ be the simplicial set whose $%
k- $simplicies are locally trivial families of partitions $(M_{0},F,M_{1})$
(parameterized by the standard simplex $\Delta ^{k}$) and let $\underline{\calP}%
(M)$ be the simplicial subset of $\calP(M)$ of those partitions $(M_{0},F,M_{1})$
satisfying that $p[\partial F]$ is the minimum of $p[F]$ where $%
p:M\times I\rightarrow I$ is the projection onto the second factor.

A partial ordering can be defined by letting $(M_{0},F,M_{1})<(M_{0}^{\prime
},F^{\prime },M_{1}^{\prime })$ if $M_{0}\subset M_{0}^{\prime }$ and the
two inclusion maps%
\begin{equation*}
F\rightarrow M_{0}^{\prime }-(M_{0}-F)\leftarrow F^{\prime }
\end{equation*}%
are homotopy equivalences. This partial ordering on $\calP(M)$ (resp. $%
\underline{\calP}(M)$) defines a simplicial partially ordered set, and hence a
simplicial category which is denoted by $h\calP(M)$ (resp. $h\underline{\calP}(M)$).
Let $h\mathcal{P}_{k}^{l,m}(M)$ (resp. $h\underline{\calP}_{k}^{l,m}(M)$) be the connected
component of $h\calP(M)$ (resp. $h\underline{\calP}(M)$) containing the particular
partition given by attaching $k$ $m-$handles trivially to $M\times \lbrack
a,a^{\prime }]$ on $M\times a^{\prime }$ in such a way that the
complementary $l-$handles are trivially attached to $M\times \lbrack
b^{\prime },b]$ on $M\times b^{\prime }$, $a<a^{\prime }<b^{\prime }<b$.

Let $d=\dim M$ and let $J$ be a closed interval and consider the limit $$\underset{k,l,m}{\underrightarrow{\lim}}h\underline{\calP}_{k}^{l,m}(M\times J^{l+m-d})$$
where the maps in the direct system are given by the lower stabilization map
\begin{equation}\label{eq:lowerstab}
\sigma _{l}:h\underline{\calP}_{k}^{l,m}(M\times J^{l+m-d})\rightarrow h%
\underline{\calP}_{k}^{l+1,m}(M\times J^{l+m+1-d})
\end{equation}
and the upper stabilization map%
\begin{equation}\label{eq:upperstab}
\sigma _{u}:h\underline{\calP}_{k}^{l,m}(M\times J^{l+m-d})\rightarrow h%
\underline{\calP}_{k}^{l,m+1}(M\times J^{l+m+1-d})
\end{equation}
which are defined in \cite[p.298]{Vogelllong}.

By a slight modification of the (anti-)involution on $h\calP(M)$ which is
induced by turning a partition upside down \cite[p.297]{Vogelllong}, one obtains an
involution $\mathscr{T}$ (up to homotopy) on $h\underline{\calP}(M)$ which satisfies $%
\sigma _{l}\circ \mathscr{T}\simeq \mathscr{T}\circ \sigma _{u}$ and $\sigma _{u}\circ \mathscr{T}\simeq
\mathscr{T}\circ \sigma _{l}$, and hence $\mathscr{T}$ induces an involution $$\underset{k,l,m}{\underrightarrow{\lim}}h\underline{\calP}_{k}^{l,m}(M\times J^{l+m-d})\rightarrow \underset{k,l,m}{\underrightarrow{\lim}}h\underline{\calP}_{k}^{l,m}(M\times J^{l+m-d}).$$

Denote by $\calP_{k}^{l,m}(M)$ (resp. $\underline{\calP}_{k}^{l,m}(M)$) the
simplicial set of the objects of the simplicial category $h\mathcal{P}_{k}^{l,m}(M)$
(resp. $h\underline{\calP}_{k}^{l,m}(M)$). Then similar to the above, one also
gets the involution on the following limits $$\underset{k,l,m}{\underrightarrow{\lim}}
\mathcal{P}_{k}^{l,m}(M\times J^{l+m-d}), \underset{l,m}{\underrightarrow{\lim}}
\calP_{0}^{l,m}(M\times J^{l+m-d}), \underset{k,l,m}{\underrightarrow{\lim}}
\underline{\calP}_{k}^{l,m}(M\times J^{l+m-d}), \underset{
l,m}{\underrightarrow{\lim}}\underline{\calP}_{0}^{l,m}(M\times J^{l+m-d}).$$

Let $\left\vert \ \  \right\vert ^{+}$ denote the space obtained by
performing the plus construction (with respect to a maximal normal perfect subgroup of the fundamental group) on the geometric realization $\left\vert
\ \  \right\vert $. (Note that if $S$ is a simplicial category, then $\left\vert S \right\vert $ denote the classifying space.) By \cite{Waldhausenmfd},
the inclusion-induced map
\begin{equation}\label{fiberation}
\left\vert \underset{\overrightarrow{k,l,m}}{\lim }\underline{\calP}%
_{k}^{l,m}(M\times J^{l+m-d})\right\vert ^{+}\rightarrow \left\vert \underset%
{\overrightarrow{k,l,m}}{\lim }h\underline{\calP}_{k}^{l,m}(M\times
J^{l+m-d})\right\vert ^{+}
\end{equation}
has a homotopy
fiber $\left\vert \underset{\overrightarrow{l,m}}{\lim }\underline{H}%
(M\times J^{l+m-d})\right\vert $ where $\underline{H}(M\times J^{l+m-d})=%
\underline{\calP}_{0}^{l,m}(M\times J^{l+m-d})$. Moreover, the map (\ref{fiberation}) is compatible with the involutions induced by $\mathscr{T}$ \cite[p.299]{Vogelllong}. Furthermore it is proved in \cite{Waldhausenmfd} that a component of $A(M)$
is homotopy equivalent to
$\left\vert \underset{\overrightarrow{k,l,m}}{\lim }h\underline{\calP}_{k}^{l,m}(M\times
J^{l+m-d})\right\vert ^{+}$
and from \cite[Corollary 2.10]{Vogelllong} the involution induced by $\mathscr{T}$ is compatible (after taking homotopy groups) with the involution $\tau_{\xi}$ on $A(M)$ corresponding to the sphere bundle $\xi$ associated to the stable tangent bundle of $M$. Therefore the map (\ref{fiberation}) induces a long exact sequence
\begin{equation}\label{eq:longexact}
{\small\xymatrix{
\cdots\ar[r]&\pi _{i+2}\left\vert \underset{\overrightarrow{k,l,m}}{\lim
}\underline{\calP}_{k}^{l,m}(M\times J^{l+m-d})\right\vert ^{+}\ar[r]^-{\alpha
^{\prime }}&\pi _{i+2}A(M)\ar[r]^-{\gamma}&\pi _{i+1}\left\vert
\underset{\overrightarrow{l,m}}{\lim }\underline{H}(M\times
J^{l+m-d})\right\vert\ar[r]&\cdots
}
}
\end{equation}
which is compatible with the involution $\tau_{\xi\ast}$ and the involutions induced by $\mathscr{T}$ on $\pi _{\ast}\left\vert \underset{\overrightarrow{k,l,m}}{\lim
}\underline{\calP}_{k}^{l,m}(M\times J^{l+m-d})\right\vert ^{+}$ and $\pi _{\ast}\left\vert
\underset{\overrightarrow{l,m}}{\lim }\underline{H}(M\times
J^{l+m-d})\right\vert.$

It turns out that $$h^{DIFF}(M):=\left\vert \underset{\overrightarrow{k,l,m}}{\lim
}\underline{\calP}_{k}^{l,m}(M\times J^{l+m-d})\right\vert ^{+}$$
is a homology theory. This is seen by first showing the corresponding functor $h^{PL}$ defined in terms of PL manifolds is a homological functor, and then using smoothing theory to obtain the desired conclusion. The proof was outlined by Waldhausen in [Wa82, Proposition 5.5]. The argument was completed by Waldhausen-Jahren-Rognes (see Proposition 1.4.8 and the subsequent discussion in \cite{WJR}). Furthermore the homology theory $h^{DIFF}(\ )$ is identified with stable homotopy theory \cite{Waldhausen83}, that is, there is a weak homotopy equivalence
\beq
h^{DIFF}(M)\simeq\Omega^{\infty}\Sigma^{\infty}(M_+).
\eeq
A more concrete description of this equivalence is given via the \textit{Gauss map} \cite[p.157]{Waldhausenmfd} (which we discuss further below)

%Furthermore, the terms in the long exact sequence can be identified with the stable homotopy groups of $M$ and the homotopy groups of the stable pseudoisotopy space.
\beq
\theta: h^{DIFF}(M)\to\Omega^{\infty}\Sigma^{\infty}(M_+)
\eeq
Note that Waldhausen \cite[p.157]{Waldhausenmfd} shows that $\theta$ is a homotopy retract, but since $M$ is compact, and the domain and the target are weakly equivalent, the map $\theta$ itself must be a weak equivalence.

Observe also that the sequence (\ref{eq:longexact}) splits by \cite[p.153]{Waldhausenmfd}. In other words there is a homomorphism
$$q^{\prime }:\pi _{i+2}A(M)\to\pi _{i+2}h^{DIFF}(M)$$
such that $q^{\prime }\circ \alpha ^{\prime }=id$.

Let now
\begin{equation*}
\alpha :=\alpha ^{\prime }\circ \theta _{\ast }^{-1}:\pi _{i+2}^{s}(M_{+})\rightarrow\pi
_{i+2}A(M).
\end{equation*}
and $H(M\times
J^{l+m-d}):=\left\vert \calP_{0}^{l,m}(M\times
J^{l+m-d})\right\vert $. The proof of \cite[p.297, Proposition 2.2]{Vogelllong} shows that there is a homotopy equivalence $$\xymatrix{\Omega H(M\times J^{2L})\ar[r]^-{\simeq}&P(M\times J^{2L})}$$
inducing isomorphisms
\begin{equation}\label{eq:h-space-pseudo}
\xymatrix{\pi _{i+1}H(M\times J^{2L})\ar[r]^-{\cong}&\pi _{i}P(M\times J^{2L})}
\end{equation}
for all $i$. Moreover, by Igusa's stabilization theorem
\cite{Igusa88}, when $L>>0$, there are isomorphisms
\begin{equation}\label{eq:h-spaceisom}
  \xymatrix{\pi _{i+1}\left\vert\underset{l,m}{\underrightarrow{\lim}}\underline{H}(M\times J^{l+m-d})\right\vert&\ar[l]_-{\cong}\pi_{i+1}\left\vert\underline{H}(M\times J^{2L})\right\vert\ar[r]^{\cong}&\pi_{i+1}H(M\times J^{2L})}
\end{equation}
and
\begin{equation}\label{eq:pseudoisotopy}
  \xymatrix{\pi _{i}P(M\times J^{2L})\ar[r]^-{\cong}&\pi_{i}\mathscr{P}(M)}
\end{equation}
The composition of the map $\gamma$ in the sequence (\ref{eq:longexact}) and the isomorphisms (\ref{eq:h-space-pseudo}), (\ref{eq:h-spaceisom}) and (\ref{eq:pseudoisotopy}) give a homomorphism $\beta:\pi_{i+2}A(M)\to\pi_{i}\mathscr{P}(M)$. All together, we have a short exact sequence
\begin{equation}\label{eq:Waldhausen-SES}
0\rightarrow \pi^{s}_{i+2}(M_+)\xrightarrow{\alpha}\pi_{i+2}A(M)%
\xrightarrow{\beta} \pi_{i}\mathscr{P}(M)\rightarrow 0
\end{equation}
for each $i\geq 0$.
This sequence splits as the sequence (\ref{eq:longexact}) splits.

The next proposition explains the behavior of the short exact sequence (\ref{eq:Waldhausen-SES}) with respect to the involutions already mentioned. Recall that given a $d-$spherical fibration $\eta\to M$, the involution $\tau_{\eta}:A(M)\to A(M)$ is defined  \cite[p.300]{Vogelllong}. For the trivial fibration $M\times S^0\to M$, this involution is denoted by $\tau_{\varepsilon}$.

\begin{proposition}\label{equivariantshortexact}
The following two conditions are satisfied :\begin{enumerate}
                                              \item[(1)] the  equation
\beq
\tau_{\varepsilon\ast}\circ\alpha=\alpha
\eeq
holds after tensoring with $\Q$.
 \item[(2)]the homomorphism $\beta$ makes the following diagram commute:$$\xymatrix{\pi_{i+2}^{\mathbb{Q}}A(M)\ar[r]^{\beta}\ar[d]_{-\tau_{\varepsilon\ast}}&\pi_{i}^{\mathbb{Q}}\mathscr{P}(M)\ar[d]^{\iota^{\ast}_{S}}\\
                                                 \pi_{i+2}^{\mathbb{Q}}A(M)\ar[r]^{\beta}&\pi_{i}^{\mathbb{Q}}\mathscr{P}(M)}$$
                                            \end{enumerate}
\end{proposition}

To prove Proposition \ref{equivariantshortexact}, we need Lemmas \ref{difference-1} and \ref{pre-equivariantshortexact} below.\footnote{We thank Kristian Moi for indicating to us the proof of Lemma \ref{difference-1}.}

\begin{lemma}\label{difference-1}
Let $M$ be a smooth oriented $d$-manifold and let $\xi\to M$ be the sphere bundle associated to the stable tangent bundle $TM\oplus\varepsilon^1$. Then the following equation holds on all rational homotopy groups of $A(M)$:
\beq
\tau_{\xi\ast}=(-1)^d\cdot\tau_{\varepsilon\ast}
\eeq
\end{lemma}
\begin{proof}
By abuse of notation, $M$ will denote both a manifold and its singular complex. Let $G=G(M)$ be the Kan loop group of $M$ and $G\to P\to M$ its universal principal fibration (in particular $P$ is contractible). Let
$\mathcal{R}_{hf}(W,G)$ be the category of $G$-simplicial sets having $W$ as a retract and satisfying that the geometric realization of every object has the $|G|$-homotopy type  of a finite $|G|$-free $CW$ complex relative to $W$ \cite[pp. 377-379]{Waldhausen100}. Waldhausen also shows that
both the category $\mathcal{R}_{hf}(M)$ of retractive spaces over $M$ and the category $\mathcal{U}_{hf}(G):=\mathcal{R}_{hf}(\ast,G)$ of pointed $G$-sets are categories with cofibrations and weak equivalences, on which homotopy equivalent models for $A(M)$ can be built. Moreover, the linearization map $A(M)\to K(\mathbb{Z} [G])$ is induced by the following composition
\begin{equation*}
L:\mathcal{R}_{hf}(M)\xrightarrow{\psi}\mathcal{R}_{hf}(P,G)\xrightarrow{\varphi}\mathcal{U}_{hf}(G)\xrightarrow{\ell}\mathcal{M}(\mathbb{Z} [G])
\end{equation*}
where $\mathcal{M}(\mathbb{Z}[G])$ is the category of simplicial $\mathbb{Z}[G]$-modules defined in Section \ref{Section:VInv}, the map $\psi$ is defined in \cite[Lemma 2.1.3]{Waldhausen100}, and the map $\varphi$ maps an object $Y$ to $Y\times_M P/M\times_MP$ \cite[p.382]{Waldhausen100}. The map $\ell$ is the ``obvious'' linearization map sending a simplicial $G$-set to the free simplicial $\mathbb{Z} [G]$-module generated by the non-basepoint elements.

Fiberwise smash product with $\xi$ defines a functor \cite[p.281]{Vogelllong}
\begin{equation*}
\xi\cdot: \mathcal{R}_{hf}(M)\to\mathcal{R}_{hf}(M)
\end{equation*}
which induces a map in $K$-theory so that the following
diagram commutes up to homotopy \cite[Proposition 2.5]{Vogelllong}
\begin{equation*}
\xymatrix{
&A(M)\ar[dr]_{\tau_{\xi}}\ar[r]^{\tau_{\varepsilon}} & A(M)\ar[d]^{\xi\cdot}\\
&&A(M)
}
\end{equation*}
Since the linearization map $L:A(M)\to K(\mathbb{Z} [G])$ is a rational equivalence, it suffices to show that $\xi\cdot$ induces, after linearization and rationalization, multiplication by $(-1)^d$ on the rational $K$-groups $K_*(\mathbb{Z}[G])\otimes\mathbb{Q}$ of the simplicial ring $\mathbb{Z}[G]$.

A direct calculation shows that the following diagram commutes
\begin{equation*}
\xymatrixcolsep{5pc}\xymatrix{
&\mathcal{R}_{hf}(M)\ar[d]^{\psi}\ar[r]^-{\xi\cdot}&\mathcal{R}_{hf}(M)\ar[d]^{\psi}\\
&\mathcal{R}_{hf}(P,G)\ar[d]^{\varphi}\ar[r]^-{\xi^P\cdot}&\mathcal{R}_{hf}(P,G)\ar[d]^{\varphi}\\
&\mathcal{U}_{hf}(G)\ar[d]^{\ell}\ar[r]^-{\wedge\Th(\xi^P)}&\mathcal{U}_{hf}(G)\ar[d]^{\ell}\\
&\mathcal{M}(\mathbb{Z}[G])\ar[r]^-{\otimes_{\mathbb{Z}}\widetilde{\mathbb{Z}}[\Th(\xi^P)]}&\mathcal{M}(\mathbb{Z}[G])\\
}
\end{equation*}
where $\xi^P=\xi\times_M P$ and $\Th(\xi^P):=\xi^P/M\times_MP$. In conclusion, fiberwise smashing with $\xi$ amounts, after linearization and rationalization, to tensoring with the simplicial module $\widetilde{\mathbb{Z}}[\Th(\xi^P)]$. Thus it suffices to show that the latter induces multiplication by $(-1)^d$ at the level of homotopy groups. In fact, let $\xi$, $M$ and $K(\mathbb{Z},d)$ be given the trivial $G$-actions. Since $(\xi^P,P)$ is a principal $G$-bundle over $(\xi,M)$, the projection $(\xi^P,P)\to(\xi,M)$ is $G$-equivariant and hence gives a $G-$equivariant map $\Th(\xi^P)\rightarrow \Th(\xi):=\xi/M$. Composing this map with a representative map $\Th(\xi)\to K(\mathbb{Z},d)$ for the Thom class of the oriented $S^d-$bundle $\xi\rightarrow M$, gives a $G-$equivariant map $\mathfrak{t}:\Th(\xi^P)\to K(\mathbb{Z},d)$
which induces a map of simplicial $\mathbb{Z}[G]-$modules$$\tilde{\mathfrak{t}}:\widetilde{\mathbb{Z}}[\Th(\xi^P)]\to\widetilde{\mathbb{Z}}[S^d]$$
where we have used that $\widetilde{\mathbb{Z}}[S^d]$ is a model for $K(\mathbb{Z},d)$.
Since $\mathfrak{t}$ represents the Thom class of the $S^d-$bundle $\xi^P$ over the contractible base $P$, $\Th(\xi^P)$ is homotopy equivalent to $S^d$ and the class of $\mathfrak{t}$ generates $H^d(\Th(\xi^P))=\mathbb{Z}$, and hence $\mathfrak{t}$ induces isomorphisms on all homotopy groups. Since the homomorphism $\pi_{\ast}(\tilde{\mathfrak{t}})$ on the homotopy groups can be identified with the isomorphism $\pi_{\ast}(\mathfrak{t})$ through Hurewicz isomorphism $\pi_{\ast}\Th(\xi^P)\cong\pi_{\ast}\widetilde{\mathbb{Z}}[\Th(\xi^P)]$, it follows that $\tilde{\mathfrak{t}}$ is a weak homotopy equivalence of simplicial $\mathbb{Z}[G]$-modules.

Therefore, by \cite[Lemma 1.3.1]{Waldhausen100}, the functors $\otimes_{\mathbb{Z}}\widetilde{\mathbb{Z}}[\Th(\xi^P)]$ and $\otimes_{\mathbb{Z}}\widetilde{\mathbb{Z}}[S^d]$ on the category $\mathcal{M}(\mathbb{Z}[G])$ induce the same map in $K$-theory. It is now an easy consequence of the additivity theorem that $\otimes_{\mathbb{Z}}\widetilde{\mathbb{Z}}[S^d]:K(\mathbb{Z}[G])\to K(\mathbb{Z}[G])$ induces multiplication by $(-1)^d$ in homotopy groups. This completes the proof of the lemma.
\end{proof}

\begin{lemma}\label{pre-equivariantshortexact}
Let $M$ and $\xi$ be as in Lemma \ref{difference-1}. Then
the following two conditions are satisfied :\begin{enumerate}
                                              \item[(1)]if $M$ is a codimension $0$ submanifold of $\mathbb{R}^d$, then the equation $\tau_{\xi\ast}\circ\alpha=(-1)^{d}\alpha$ holds.
                                              \item[(2)]the homomorphism $\beta$ makes the following diagram commute:$$\xymatrix{\pi_{i+2}A(M)\ar[r]^{\beta}\ar[d]_{(-1)^{d+1}\tau_{\xi\ast}}&\pi_{i}\mathscr{P}(M)\ar[d]_{\iota_{\ast}^{S}}\\
                                                 \pi_{i+2}A(M)\ar[r]^{\beta}&\pi_{i}\mathscr{P}(M)}$$
                                            \end{enumerate}
\end{lemma}
\begin{proof}
First observe that the isomorphisms (\ref{eq:h-space-pseudo}) anti-commute with the involutions, namely, the following diagram commutes for all $i$:
$$\xymatrix{\pi _{i+1}H(M\times J^{2L})\ar[r]^-{\cong}\ar[d]&\pi _{i}P(M\times J^{2L})\ar[d]^-{-\tau_{\ast}}\\\pi _{i+1}H(M\times J^{2L})\ar[r]^-{\cong}&\pi _{i}P(M\times J^{2L})}$$
where the left vertical map is the involution induced by $\mathscr{T}$.

The isomorphisms (\ref{eq:h-spaceisom})
are compatible with the involutions, and the stabilization induces an isomorphism
\begin{equation}
  \xymatrix{\pi _{i}P(M\times J^{2L})\ar[r]^-{\cong}&\pi_{i}\mathscr{P}(M)}
\end{equation}
which is compatible with involutions $\iota_{\ast}$ and $\iota_{\ast}^S$ up to $(-1)^{\dim M}$.
Condition (2) now follows since $\gamma$ is compatible with the involution $\tau_{\xi}$.

It remains to show the condition (1), namely $\tau_{\xi*}\circ\alpha=(-1)^{d}\alpha$ when $M$ is a codimension $0$ submanifold of $\mathbb{R}^d$. Since $\alpha =\alpha ^{\prime }\circ \theta
_{\ast }^{-1}$ and $\alpha ^{\prime }$ is compatible with the involutions $\tau_{\xi\ast}$ and the induced involution $\mathscr{T}_{\ast}$ on
%$\pi_{i+2}\left\vert\underset{\overrightarrow{k,l,m}}{\lim }%
%\underline{\calP}_{k}^{l,m}(M\times J^{l+m-d})\right\vert$,
$\pi_{i+2}h^{DIFF}(M)$, it suffices to show that the following diagram commutes:
\begin{equation}\label{eq:thetadiag}
\xymatrix{
\pi _{i+2}h^{DIFF}(M)\ar[r]^-{\theta_{\ast}}\ar[d]^-{\mathscr{T}_{\ast}}&\pi _{i+2}^s(M_{+})\ar[d]^-{\cdot(-1)^d}\\\pi _{i+2} h^{DIFF}(M)\ar[r]^-{\theta_{\ast}}&\pi _{i+2}^s(M_{+})}
\end{equation}

To see this we need to understand how the map
\begin{equation*}
\xymatrix{\theta :h^{DIFF}(M)\ar[r]^-{\simeq}& \Omega ^{\infty
}\Sigma ^{\infty }(M_{+})}
\end{equation*}
appears \cite[pp.155-157]{Waldhausenmfd}.

First consider the stabilization maps  $$\Sigma:\Map(Y/\partial Y, S^{L})\rightarrow \Map(Y\times J/\partial(Y\times J), S^{L+1})$$
induced by the homeomorphism $Y\times J/\partial(Y\times J)\approx Y\wedge J/\partial J$ and the structural map $S^{L}\wedge J/\partial J\rightarrow S^{L+1}$, where $L=l+m$ and $Y=M\times J^{L-d}$. Let $$r_{\ast}:\Map(Y\times J/\partial(Y\times J), S^{L+1})\rightarrow \Map(Y\times J/\partial(Y\times J), S^{L+1})$$ be the map induced by the reflection $r:S^{L+1}\rightarrow S^{L+1}:(x_0, \cdots, x_{L}, x_{L+1})\mapsto (x_0, \cdots, x_{L}, -x_{L+1})$. Denote $M\times J^{l+m-d}/\partial(M\times J^{l+m-d})$ by $M\times J^{l+m-d}/\partial$. Then identify $\Omega ^{\infty
}\Sigma ^{\infty }(M_{+})$ with $\underset{l,m}{\underrightarrow{\lim}}\,\Map(M\times J^{l+m-d}/\partial,S^{l+m})$ (up to homotopy) via Poincar\'e duality, where the direct limit system is given by $\Sigma$ and $r_*\circ \Sigma$.

In \cite[p.155,p.183]{Waldhausenmfd}, Waldhausen defines the space $Q^{d}$ of germs of normally oriented $d-$planes in $\mathbb{R}^{d+1}$ and gives an explicit homotopy equivalence $Q^d\simeq S^d$. Moreover, he gives a bundle map between the trivial bundles over the $N-$skeleton of $\left\vert\Dd\right\vert$ with fibers $M\times [a,b]$ and $Q^d$ where $N=\min\{m,l\}-2$ (c.f. \cite[p.156]{Waldhausenmfd}). This bundle map combined with the homotopy equivalence $Q^d\simeq S^d$ induces a bundle map $\Theta$ between the trivial bundles over the $N-$skeleton of $\left\vert\Dd\right\vert$ with fibers $M\times [a,b]$ and $S^d$, where the bundle map $\Theta$ is trivial near $\partial M\times [a,b]$. In particular, over each $0-$simplex $(M_1,F,M_2)$ of $\left\vert\Dd\right\vert$, the bundle map $\Theta$ restricts to a map
$$M\times [a,b]\rightarrow S^d$$
which extends the Gauss map $F\rightarrow S^d$ for the embedding $F\subset M\times [a,b]\subset\mathbb{R}^{d+1}$.

By restricting $M\times [a,b]$ to $M\times t$ for $t\in[a,b]$, the bundle map $\Theta$ provides a continuous family of maps over the $N-$skeleton
$$\Theta_t: \left\vert\Dd\right\vert\rightarrow \Map(M/\partial M, S^d).$$
In particular $\Theta_a$ is homotopic to $\Theta_b$. More generally, a continuous family of maps from $\left\vert\D\right\vert$ to $\Map(M\times J^{l+m-d}/\partial(M\times J^{l+m-d}), S^{l+m})$ is defined over the $N-$skeleton similarly and also denoted by $\Theta_t$ for $t\in[a,b]$.
The map $\theta$ we are looking for will essentially be the map obtained from  $\Theta_a$ (i.e. when $t=a$) after passing to limits (w.r.t. $l$ and $m$). To be more precise, note that  $\Sigma\circ \Theta_a\sim\Theta_a\circ \sigma_l$ \cite[p.156]{Waldhausenmfd}.
On the other hand we will show in Lemma \ref{lemma:wl} below, that the diagram
\begin{equation}\label{eq:Sigma-r-stab}
\xymatrix{\left\vert \underline{\calP}%
_{k}^{l,m}(M\times J^{l+m-d})\right\vert\ar[r]^-{\Theta_a}\ar[d]^-{\sigma_u}&\Map(M\times J^{l+m-d}/\partial,S^{l+m})\ar[d]^-{r_{\ast}\circ\Sigma}\\ \left\vert \underline{\calP}%
_{k}^{l,m+1}(M\times J^{l+m+1-d})\right\vert\ar[r]^-{\Theta_{a}}&\Map(M\times J^{l+m+1-d}/\partial,S^{l+m+1})}
\end{equation}
commutes. Thus we obtain the desired map\footnote{Since $\pi_{1}\left\vert \underset{\overrightarrow{k,l,m}}{\lim }\underline{\calP}_{k}^{l,m}(M\times J^{l+m-d})\right\vert\cong \pi_{1}^{s}(M_+)$ is an abelian group, then $$h^{DIFF}(M)=\left\vert \underset{\overrightarrow{k,l,m}}{\lim }\underline{\calP}
_{k}^{l,m}(M\times J^{l+m-d})\right\vert^{+}=\left\vert \underset{\overrightarrow{k,l,m}}{\lim }\underline{\calP}
_{k}^{l,m}(M\times J^{l+m-d})\right\vert.$$}
$$h^{DIFF}(M)=\left\vert\underset{\overrightarrow{k,l,m}}{\lim }\underline{\calP}%
_{k}^{l,m}(M\times J^{l+m-d})\right\vert\rightarrow\underset{l,m}{\underrightarrow{\lim}}\,\Map(M\times J^{l+m-d}/\partial,S^{l+m})\simeq\Omega^{\infty}
\Sigma^{\infty}(M_+)$$
%The map $\Theta_a$ induces a map\textbf{(?How about $k$)}
%where $M\times J^{l+m-d}/\partial$ denotes $M\times J^{l+m-d}/\partial(M\times J^{l+m-d})$. Since $$\underset{l,m}{\underrightarrow{\lim}}\,\Map(M\times J^{l+m-d}/\partial,S^{l+m})\simeq\Omega ^{\infty }\Sigma ^{\infty }(M_{+})$$ by \textbf{Poincar\'{e} duality(?)}, the map $\Theta_a$ approximates $\theta$.
The commutativity of the diagram (\ref{eq:thetadiag}) follows directly from Lemmas \ref{lemma:tran} and \ref{lemma:inter}, which we prove below separately.
\end{proof}

\begin{lemma}\label{lemma:tran}The diagram
\begin{equation}\xymatrix{\left\vert \underline{\calP}%
_{k}^{l,m}(M\times J^{l+m-d})\right\vert\ar[r]^-{\Theta_a}\ar[d]^-{\mathscr{T}}&\Map(M\times J^{l+m-d}/\partial,S^{l+m})\ar[d]^-{(r\circ\rho)_{\ast}}\\ \left\vert \underline{\calP}%
_{k}^{m,l}(M\times J^{l+m-d})\right\vert\ar[r]^-{\Theta_{a}}&\Map(M\times J^{l+m-d}/\partial,S^{l+m})}
\end{equation}
commutes up to homotopy where $\rho:S^{l+m}\rightarrow S^{l+m}$ is the antipodal map and $r:S^{l+m}\rightarrow S^{l+m}$ is the reflection given by $(x_0, \cdots, x_{l+m-1}, x_{l+m})\mapsto (x_0, \cdots, x_{l+m-1}, -x_{l+m})$.
\end{lemma}
\begin{proof}
To ease the notation, we only show the case when $l+m=d$. Let $\lambda$ be the tautological bundle over $\left\vert \underline{\calP}%
_{k}^{l,m}(M)\right\vert$ with fiber $F$ over each partition $(M_1, F, M_2)$. Then $\lambda$  can be viewed as a subbundle of the trivial $M\times [a,b]$-bundle over $\left\vert \underline{\calP}%
_{k}^{l,m}(M)\right\vert$. Since $\Theta$ restricts to the fiberwise Gauss map on $\lambda$, then the diagram
\begin{equation}\label{diag:bundle}\xymatrix{\left\vert \underline{\calP}%
_{k}^{l,m}(M)\right\vert\times M\times [a,b]\ar[r]^-{\Theta}\ar[d]^-{\bar{\mathscr{T}}}&\left\vert \underline{\calP}%
_{k}^{l,m}(M)\right\vert\times S^d\ar[d]^-{\mathscr{T}\times(r\circ\rho)}\\ \left\vert \underline{\calP}%
_{k}^{m,l}(M)\right\vert\times M\times [a,b]\ar[r]^-{\Theta}&\left\vert \underline{\calP}%
_{k}^{m,l}(M)\right\vert\times S^d}
\end{equation}
commutes when restricted on $\lambda$ and hence, by an obstruction theory argument, the diagram \label{diag:bundle} commutes up to fiberwise homotopy relative to $\partial M\times[a,b]$ over the $N-$skeleton, where $\bar{\mathscr{T}}$ is the bundle map covering $\mathscr{T}$ and flipping each fiber. Restricting to $\left\vert \underline{\calP}%
_{k}^{l,m}(M)\right\vert\times M\times a$, the commutativity of this diagram implies that $(r\circ\rho)_{\ast}\circ\Theta_a$ is homotopic to $\Theta_b\circ\mathscr{T}$. The conclusion of the lemma follows since $\Theta_a$ is homotopic to $\Theta_b$.
\end{proof}
\begin{lemma} The map\label{lemma:inter}
$(r\circ\rho)_{\ast}$ induces the multiplication by $(-1)^d$ on $\pi_{\ast}\underset{\overrightarrow{l,m}}{\lim }\Map(M\times J^{l+m-d}/\partial,S^{l+m})$
\end{lemma}
\begin{proof}
First note that since $r\circ\rho:S^{d+2n}\rightarrow S^{d+2n}$ has degree $(-1)^d$, it follows that $(r\circ\rho)_*$ induces the multiplication by $(-1)^d$ on $\pi_{i}\Map(M\times J^{2n}/\bd,S^{d+2n})=\pi^{d+2n}(\Sigma^{i+2n}(M_{+}))$ for $n\gg0$ \cite[p.211, Theorem 5.2]{Hu}.
The result follows by noticing that
$\underset{\overrightarrow{l,m}}{\lim }\Map(M\times J^{l+m-d}/\partial,S^{l+m})$ is identified with the mapping telescope of the sequence of maps
\begin{equation}
\Map(M,S^d)\xrightarrow{(r_*\circ\Sigma)\circ\Sigma}
\Map(M\times J^2/\bd, S^{d+2})\xrightarrow{(r_*\circ\Sigma)\circ\Sigma}\cdots\end{equation}
\end{proof}
\begin{lemma}\label{lemma:wl}
The diagram
(\ref{eq:Sigma-r-stab})
commutes up to homotopy, where $\Theta_a$ is defined in the stable range and $r_*$ is induced by the reflection $r:S^{l+m+1}\rightarrow S^{l+m+1}, (x_0,\cdots,x_{l+m},x_{l+m+1})\mapsto(x_0,\cdots,x_{l+m},-x_{l+m+1})$
\end{lemma}
\begin{proof}
Combining the fact that $\Sigma\circ \Theta_a\sim\Theta_a\circ \sigma_l$ \cite[p.156]{Waldhausenmfd}, that $\sigma_l\circ\mathscr{T}\sim \mathscr{T}\circ\sigma_u$, and Lemma \ref{lemma:tran}, one has $(r\circ\rho)_*\circ\Theta_a\circ\sigma_u\sim\Sigma\circ(r'\circ\rho')_*\circ\Theta_a$ where $\rho,\rho'$ are the antipodal maps of $S^{l+m+1}$ and $S^{l+m}$, respectively, and $r':S^{l+m}\rightarrow S^{l+m}$ is the reflection given by $(x_0,\cdots,x_{l+m-1},x_{l+m})\mapsto(x_0,\cdots,x_{l+m-1},-x_{l+m})$. Since $(r\circ\rho)_*\circ\Sigma\circ(r'\circ\rho')_*\sim r_*\circ\Sigma$, the result follows.
\end{proof}

\begin{proof}[\textbf{Proof of Proposition \ref{equivariantshortexact}}]
By Lemmas \ref{difference-1} and \ref{pre-equivariantshortexact}, the statement (2) holds and so does the statement (1) in the case that $M$ is a codimension $0$ submanifold in $\mathbb{R}^{d}$. It remains to show the statement (2) in the case when $M$ is not a codimension $0$ submanifold in $\mathbb{R}^{d}$. By Whitney's embedding theorem, the manifold $M$ can be embedded into $\mathbb{R}^{d+N}$ for $N>>0$. Let $N(M)$ be a regular tubular neighborhood of $M$ in $\mathbb{R}^{d+N}$. We have the following commutative diagram
\begin{equation}\label{eq:ASA}
\begin{array}{ccc}
A^{S}(M) & \rightarrow  & A(M) \\
\downarrow  &  & \downarrow \eta  \\
A^{S}(N(M)) & \rightarrow  & A(N(M))%
\end{array}%
\end{equation}
with $A^{S}(Y)=\underset{\overrightarrow{k,l,m}}{\lim }\underline{\calP}%
_{k}^{l,m}(Y\times J^{l+m-\dim Y})$,
and the
horizontal maps are induced from the natural inclusions and the vertical
maps are essentially defined
by
\begin{equation*}
(M,F,N)\mapsto (p^{-1}[M],p^{-1}[F],p^{-1}[N])
\end{equation*}
(with a technical modification, c.f. \cite[p.175]{Waldhausenmfd}) where $p=\pi\times id_{I}:N(X)\times I\rightarrow X\times I$ with $\pi$ the normal bundle projection. By \cite[Proposition 5.4]{Waldhausenmfd}, the map $\eta$ can be identified (up to homotopy) with the map on $K$-theories induced by the functor $R_{hf}(M)\rightarrow R_{hf}(N(M))$ given by pushout with the inclusion $M\to N(M)$. Since the inclusion $M\to N(M)$ is a homotopy equivalence, the map $\eta$ is also an equivalence. Moreover, as a consequence of the naturality of $\tau_{\varepsilon}$ \cite[Proposition 1.17]{Vogelllong}, the map $\eta$ is
compatible with the involution $\tau_{\varepsilon}$. Since $N(M)$ is a codimension $0$ submanifold of $\mathbb{R}^{N}$, the bottom horizontal map in the diagram (\ref{eq:ASA}) induces the homomorphisms $\overline{\alpha}: \pi_{\ast}A^S(N(M))\rightarrow \pi_{\ast}A(N(M))$ satisfying $\tau_{\varepsilon}\circ\overline{\alpha}
=\overline{\alpha}$ after tensoring with $\Q$. These facts, together with the commutative diagram (\ref{eq:ASA}) imply that the statement (1) holds for general $M$ . This completes the proof.
\end{proof}

\begin{proof}[\textbf{Proof of Theorem \ref{formula}}]
The theorem follows by putting together Theorem \ref{AthS1}, Proposition \ref{equivariantshortexact}
and the following calculations of Waldhausen \cite[P.48]{WaldhausenI1978} and Farrell-Hsiang \cite{FH}
$$\Inv_{\tau_{\varepsilon\ast}}^{+}\pi_{i+2}^{\mathbb{Q}}A(\ast)=\Inv_{\tau_{BF\ast}}^{+}K_{i+2}(\mathbb{Z})\otimes%
\mathbb{Q}=0$$ for all $i\geq 0$ and
$$\dim \Inv_{\tau_{\varepsilon\ast}}^{-}\pi_{i+2}^{\mathbb{Q}}A(\ast)=\dim \Inv_{\tau_{BF\ast}}^{-}K_{i+2}(\mathbb{Z%
})\otimes\mathbb{Q}=\left\{
      \begin{array}{ll}
        1, & \mbox{if }i\equiv3\mod4\mbox{;} \\
        0, & \mbox{otherwise,}
      \end{array}
    \right.$$
\end{proof}

\section{Application to spaces of nonnegatively curved metrics}
Belegradek, Farrell and Kapovitch proved in \cite{BFK} that there exist simply connected complete nonnegatively curved manifolds $V$, such that for certain positive integers $i$ the rational homotopy groups $\pi^{\mathbb{%
Q}}_{i}\mathcal{R}_{K \geq 0}(V)$ are nontrivial, where $\mathcal{R}_{K \geq 0}(V)$ is the space of complete nonnegatively curved Riemannian metrics on $V$. For
example, when $d\geq2$ and for some explicit $i\geq 2$, they prove that there exists an $m$ such that
\begin{equation*}
\pi^{\mathbb{Q}}_{i}\mathcal{R}_{K \geq 0}(TS^{2d}\times S^{m})\neq0
\end{equation*}
where $TS^{2d}$ denotes the tangent bundle of $S^{2d}$. However, the methods in \cite{BFK} are insufficient to determine $m$ exactly when $i$ is given. It turns out, though, that this can be fixed by computing the ranks of $\Inv^{+}\pi_{k}^{\mathbb{Q}}\mathscr{P}(M)$ and $\Inv^{-}\pi_{k}^{\mathbb{Q}}\mathscr{P}(M)$ when $M$ is the unit sphere bundle of $S^{2d}$. For the sake of concreteness, we will
focus on the case $V=TS^{2d}\times S^{m}$. Other manifolds $V$ should be treated similarly.

For a graded $\mathbb{Q}$--vector space $A_{\ast}=\underset{i}{\oplus}A_{i}$, let $$\textbf{p}(A_{\ast})=\underset{i}{\sum}(\dim_{\mathbb{Q}}A_{i})\cdot t^{i}$$ be the Poincar\'{e} Series of $A_{\ast}$. In this section, we will compute the Poincar\'{e} Series for $\Inv^{\pm}H_{S^{1}}^{\ast }(LM;\mathbb{Q})$ and $\Inv^{\pm}\pi _{\ast }^{\mathbb{Q}}\mathscr{P}(M)$. We begin with the following lemma.

\begin{lemma}\label{lemma:HS1}
Let $X$ be a finite complex and let $H^{*}_{S^{1}}(LX;\mathbb{Q}):=H^{*}(ES^{1}\times_{S^{1}}LX;\mathbb{Q})$. Then
\begin{enumerate}
\item[(1)]$\textbf{p}(\Inv^{\pm}H_{*}^{S^{1}}(LX,\ast;\mathbb{Q}))=\textbf{p}(\Inv^{\pm}H_{*}^{S^{1}}(LX;\mathbb{Q}))-\textbf{p}(\Inv^{\pm}H_{*}^{S^{1}}(\ast;\mathbb{Q}))$
\item[(2)]$\Inv^{\pm}H_{*}^{S^{1}}(LX;\mathbb{Q})=\Inv^{\pm}H^{*}_{S^{1}}(LX;\mathbb{Q})$
\item[(3)]$\textbf{p}(\Inv^{+}H_{*}^{S^{1}}(\ast;\mathbb{Q}))=\frac{1}{1-t^{4}}$ and $\textbf{p}(\Inv^{-}H_{*}^{S^{1}}(\ast;\mathbb{Q}))=\frac{t^{2}}{1-t^{4}}$
\end{enumerate}
\end{lemma}
\begin{proof}The constant map $X\rightarrow \ast$ induces a retraction $ES^{1}\times_{S^{1}}LX\rightarrow ES^{1}\times_{S^{1}}\ast$ which is compatible with the involution $T$, hence induces the decomposition $$H_{\ast}^{S^{1}}(LX;\mathbb{Q})\cong H_{\ast}^{S^{1}}(LX,\ast;\mathbb{Q})\oplus H_{\ast}^{S^{1}}(\ast;\mathbb{Q}) $$ which is compatible with the involution $T_{\ast}$ as well and so this implies (1).
Formula (2) follows from the Universal Coefficient theorem.
Since $H^{*}_{S^{1}}(\ast;\mathbb{Q})=H^{*}(BS^{1};\mathbb{Q})=\mathbb{Q}[\alpha]$ with $\dim \alpha=2$ and the involution $T_{\ast}$ is given by $\alpha \mapsto -\alpha$ \cite[Theorem 3.3]{KS}, then (2) implies (3).

\end{proof}
Let $\Lambda (x_{1},\dots, x_{l})$ denote the free graded algebra over $\mathbb{Q}$ generated by $x_{1},\dots, x_{l}$; this algebra is the tensor product of the polynomial algebra generated by the even dimensional generators and the exterior algebra generated by the odd dimensional generators.
\begin{example}\label{example:compute sphere}
  Let $M=S(TS^{2d})$: the unit tangent bundle of $S^{2d}$ for $%
d\geq 2$. The rational cohomology ring for $M$
is the exterior algebra $\Lambda (a)$ on one generator $a$ with $\dim a=4d-1$%
. Then the minimal model for $M$ is $(\Lambda (x),\delta)$ where differential $\delta=0$ and $\dim x=4d-1 $. By \cite{VPB},
the minimal
model for the space $ES^{1}\times _{S^{1}}LM$ is $(\Lambda (\alpha ,x,\bar{x}),D) $, $\dim \alpha =2$, $\dim x=4d-1$, $\dim \bar{x}=4d-2$ with differential $%
D\alpha =0$, $D\bar{x}=0$ and $Dx=\alpha \bar{x}$. It is not hard to check
that $H_{S^{1}}^{\ast }(LM;\mathbb{Q})$ is
generated freely as vector space by $\alpha ^{k},\bar{x}^{l}$ for all nonnegative integers $k,l$.
As the involutions on $H_{S^{1}}^{\ast }(LM;\mathbb{Q})$ is given by $%
\alpha \mapsto -\alpha ,\ \bar{x}\mapsto -\bar{x}$ \cite[Theorem 3.3]{KS}, then
\begin{equation*}
\textbf{p}(H_{S^{1}}^{\ast }(LM))=\frac{1}{1-t^{2}}+\frac{t^{4d-2}}{1-t^{4d-2}}
\end{equation*}%
\begin{equation*}
\textbf{p}(\Inv^{+}H_{S^{1}}^{\ast }(LM))=\frac{1}{1-t^{4}}+\frac{t^{8d-4}}{1-t^{8d-4}}
\end{equation*}%
\begin{equation*}
\textbf{p}(\Inv^{-}H_{S^{1}}^{\ast }(LM))=\frac{t^{2}}{1-t^{4}}+\frac{t^{4d-2}}{1-t^{8d-4}%
}
\end{equation*}
By the Theorem \ref{formula} and Lemma \ref{lemma:HS1}, one obtains
\begin{eqnarray*}
\textbf{p}(\Inv^{+}\pi _{\ast }^{\mathbb{Q}}\mathscr{P}(M)) &=&\frac{t^{3}}{1-t^{4}}+%
\frac{t^{8d-5}}{1-t^{8d-4}} \\
\textbf{p}(\Inv^{-}\pi _{\ast }^{\mathbb{Q}}\mathscr{P}(M)) &=&\frac{t^{12d-7}}{1-t^{8d-4}%
}
\end{eqnarray*}
\end{example}
For a compact smooth manifold $N$ let $P'(N)$ be the topological group of the diffeomorphisms of $N\times [0,1]$ that are identity on a neighborhood of $N\times\{0\}\cup\partial N\times [0,1]$. Assume further that $\partial N\neq\emptyset$ and identify $\partial N\times [0,1]$ with a fixed collar neighborhood of $N$. Define $\iota_{N}:P'(\partial N)\rightarrow \Diff(N)$ to be the map that extends every $f\in P'(\partial N)$ from the collar $\partial N\times [0,1]$ to a diffeomorphism of $N$ by taking the identity outside the collar. It follows from \cite[Theorems 1.1 and 1.2]{BFK} that $\pi^{\mathbb{Q}}_{i+1}\mathcal{R}_{K \geq 0}(TS^{2d} \times S^{m})\neq0$ if $\ker\pi_{i}^{\mathbb{Q}}(\iota_{E\times
S^{m}})\neq0$ where $d\geq2$ and $E$ is the associated disk bundle of $TS^{2d}$. We will find a condition for $\ker\pi_{i}^{%
\mathbb{Q}}(\iota_{E\times S^{m}})\neq0$ in terms of the positive and negative eigenspaces of the involution on $\pi_{\ast}\mathscr{P}(\partial E)$. For this, we restate \cite[Lemma 9.4]{BFK} as follows.

\begin{lemma}\label{lemma:restate}
Let $E$ be a
compact manifold, $m,i$ be integers such that $m\geq0$, $i\geq1$, $\dim\partial E+m\geq \max\{3i+7,2i+9\}$ and
\begin{equation*}
  \frac{\dim\pi_{i}^{\mathbb{Q}}\mathscr{P}(\partial E)}{2}\leq\left\{
  \begin{array}{ll}
    \dim \Inv^{+}\pi_{i}^{\mathbb{Q}}\mathscr{P}(\partial E), & \hbox{if $\dim\partial
E+m\equiv0\mod2$;} \\
    \dim \Inv^{-}\pi_{i}^{\mathbb{Q}}\mathscr{P}(\partial E), & \hbox{otherwise.}
  \end{array}
\right.
\end{equation*}
Then
\begin{equation*}
\dim \ker\pi_{i}^{\mathbb{Q}}(\iota_{E\times S^{m}})\geq\frac{\dim\pi_{i}^{\mathbb{Q}}\mathscr{P}(\partial E)}{2}-\dim\pi_{i}^{\mathbb{Q}%
}\Diff(E\times D^{m},\partial)
\end{equation*}

\end{lemma}

\begin{proof}Let $\eta_{i}:\pi_{i}^{\mathbb{Q}}P(\partial E\times D^{m})\rightarrow \pi_{i}^{\mathbb{Q}}P(\partial E\times D^{m})$ be given by $$\eta_{i}(x)=x+\iota_{\ast}(x)$$ with $\iota$ the involution on $P(\partial E\times D^{m})$. Since $\ker\eta_{i}=\Inv^{-}\pi_{i}^{\mathbb{Q}}P(\partial E\times D^{m})$ and $$\dim \pi_{i}^{\mathbb{Q}}P(\partial E\times D^{m})=\dim \mbox{Im}\eta_{i}+\dim \ker\eta_{i},$$ then by assumption one gets
\begin{equation}\label{imeta}
  \dim \mbox{Im}\eta_{i}=\dim \Inv^{+}\pi_{i}^{\mathbb{Q}}P(\partial E\times D^{m})\geq \frac{\dim\pi_{i}^{\mathbb{Q}}\mathscr{P}(\partial E)}{2}
\end{equation}
Since the inclusion $P'(M)\rightarrow P(M)$ is a weak homotopy equivalence (c.f.\cite[Chapter 1, Proposition 1.3]{Igusa88}), then the arguments in \cite[p.11]{BFK} imply that the image of the $\pi_{i}^{\mathbb{Q}}$--homomorphism induced by the inclusion $j: P'_{\partial}(\partial E\times D^{m})\rightarrow P'(\partial E\times D^{m})$ has dimension not less than $\dim \mbox{Im}\eta_{i}$, hence by the inequality (\ref{imeta}) we have $$\dim\mbox{Im}\pi^{\mathbb{Q}}_{i}(j)\geq\frac{\dim\pi_{i}^{\mathbb{Q}}\mathscr{P}(\partial E)}{2}.$$ Then the same argument as in the proof of \cite[Lemma 9.4]{BFK} completes the proof.

\end{proof}

\begin{example} Let $E$ be the associated disk bundle of $TS^{2d}$. By \cite[Corollary 8.5, Proposition 9.1 and Theorem
9.11]{BFK}, a sufficient condition for $\dim\pi _{i}^{\mathbb{Q}%
}\Diff(E\times D^{m},\partial )=0$ and $\dim\pi _{i}^{\mathbb{Q}}\mathscr{P}%
(\partial E)=1$ is as follows: $$\left\{
                     \begin{array}{ll}
                       i=8d-5+(4d-2)j\ \mbox{for some odd}\ j\geq 1\\
                       3i+9<4d+m\\
                       m+i\geq 4d\\
                       m\equiv 3\ \mbox{or}\ 2d\mod4
                     \end{array}
                   \right.$$
This, together with the condition Lemma \ref{lemma:restate} gives a sufficient
condition for $\ker\pi_{i}^{\mathbb{Q}}(\iota _{E\times S^{m}}) \neq 0$,
which is
\begin{equation}\label{ex1cond}
  \left\{
  \begin{array}{l}
    i=8d-5+(4d-2)j\ \mbox{for some odd}\ j\geq 1\\
    3i+9<4d+m\\
    m+i\geq 4d\\
    m\equiv 3\ \mbox{or}\ 2d\mod4\\
     \frac{\dim\pi_{i}^{\mathbb{Q}}\mathscr{P}(\partial E)}{2}\leq\left\{
  \begin{array}{ll}
    \dim \Inv^{+}\pi_{i}^{\mathbb{Q}}\mathscr{P}(\partial E), & \hbox{if $\dim\partial
E+m\equiv0\mod2$;} \\
    \dim \Inv^{-}\pi_{i}^{\mathbb{Q}}\mathscr{P}(\partial E), & \hbox{otherwise.}
  \end{array}
\right.
  \end{array}
\right.
\end{equation}

By the calculation
of Example \ref{example:compute sphere} we get
\begin{equation*}
\textbf{p}(\Inv^{+}\pi _{\ast }^{\mathbb{Q}}\mathscr{P}(\partial E))-\textbf{p}(\Inv^{-}\pi _{\ast }^{%
\mathbb{Q}}\mathscr{P}(\partial E))=\frac{t^{3}}{1-t^{4}}+\underset{m}{\sum}(-1)^{m}t^{(4d-2)m+8d-5}
\end{equation*}
Since $(4d-2)m+8d-5\equiv 1\mod4$ when $m$ is odd, then a necessary and
sufficient condition for $\dim \Inv^{-}\pi _{i}^{\mathbb{Q}}\mathscr{P}%
(\partial E)>\dim \Inv^{+}\pi _{i}^{\mathbb{Q}}\mathscr{P}(\partial E)$ is $%
i=(4d-2)m+8d-5$ for some odd $m$. Consequently, the condition (\ref{ex1cond}) can be
simplified as

\begin{equation}\label{example:condition sphere}
\left\{
\begin{array}{l}
i=8d-5+(4d-2)j\text{ for some odd }j\geq 1 \\
m>20d-6+(12d-6)j \\
m\equiv 2d\mod4%
\end{array}%
\right.
\end{equation}
For example, when $d=2$, the first $i$ and $m$ appearing here are $(i,m)=(17,56), (29,92), (41,128)$, which give $\pi^{\mathbb{Q}}_{18}\mathcal{R}_{K \geq 0}(TS^{4} \times S^{56})\neq0,\ \pi^{\mathbb{Q}}_{30}\mathcal{R}_{K \geq 0}(TS^{4} \times S^{92})\neq0,\ \pi^{\mathbb{Q}}_{42}\mathcal{R}_{K \geq 0}(TS^{4} \times S^{128})\neq0$.
\end{example}

{\footnotesize \ }

{\footnotesize \
\bibliographystyle{alpha}
\bibliography{bib}
}
\Addresses
\end{document}